\documentclass[11pt,a4paper]{amsart}

\RequirePackage[OT1]{fontenc}
\RequirePackage{amsthm,amsmath}
\RequirePackage[numbers]{natbib}
\RequirePackage[colorlinks,citecolor=blue,urlcolor=blue]{hyperref}

\usepackage{tikz}
\usetikzlibrary{positioning}
\usepackage{enumitem}

\usepackage{graphicx}

\newtheorem{theorem}{Theorem}
\newtheorem{lemma}[theorem]{Lemma}
\newtheorem*{lemma*}{Lemma}
\newtheorem{proposition}[theorem]{Proposition}

\newtheorem*{fact*}{Fact}
\theoremstyle{definition}
\newtheorem{remark}[theorem]{Remark}
\newtheorem{definition}[theorem]{Definition}
\newtheorem*{example*}{Example}

\newtheorem{question*}[theorem]{Question}
\newtheorem{questions*}[theorem]{Questions}

\usepackage[utf8]{inputenc}
\usepackage[english]{babel}
\usepackage{latexsym}
\usepackage{amssymb}
\usepackage{amsmath}

\def\vint_#1{\mathchoice%
      {\mathop{\kern 0.2em\vrule width 0.6em height 0.69678ex depth -0.58065ex
              \kern -0.8em \intop}\nolimits_{\kern -0.4em#1}}%
      {\mathop{\kern 0.1em\vrule width 0.5em height 0.69678ex depth -0.60387ex
              \kern -0.6em \intop}\nolimits_{#1}}%
      {\mathop{\kern 0.1em\vrule width 0.5em height 0.69678ex depth -0.60387ex
              \kern -0.6em \intop}\nolimits_{#1}}%
      {\mathop{\kern 0.1em\vrule width 0.5em height 0.69678ex depth -0.60387ex
              \kern -0.6em \intop}\nolimits_{#1}}}
\def\vintslides_#1{\mathchoice%
      {\mathop{\kern 0.1em\vrule width 0.5em height 0.697ex depth -0.581ex
              \kern -0.6em \intop}\nolimits_{\kern -0.4em#1}}%
      {\mathop{\kern 0.1em\vrule width 0.3em height 0.697ex depth -0.604ex
              \kern -0.4em \intop}\nolimits_{#1}}%
      {\mathop{\kern 0.1em\vrule width 0.3em height 0.697ex depth -0.604ex
              \kern -0.4em \intop}\nolimits_{#1}}%
      {\mathop{\kern 0.1em\vrule width 0.3em height 0.697ex depth -0.604ex
              \kern -0.4em \intop}\nolimits_{#1}}}

\usepackage{fullpage}

\def\Z{\mathbf{Z}}
\def\reals{\mathbf{R}}
\def\N{\mathbf{N}}
\def\R{\reals}

\def\D{\mathbb{D}}

\def\C{\mathbf{C}}

\def\E{\mathbb{E}\,} 
\def\P{\mathbb{P}}
\def\Prob{\mathbb{P}}
\def\1{\mathbf{1}}

\renewcommand{\Re}{\operatorname{Re}}
\renewcommand{\Im}{\operatorname{Im}}

\newcommand{\weaklyprob}{\overset{w^*}{\underset{\Prob}{\longrightarrow}}} 
 
\newcommand{\inprob}{\underset{\Prob}{\longrightarrow}} 

\numberwithin{equation}{section}
\numberwithin{theorem}{section}

\def\enddef{\ensuremath{\hfill\diamond}}

\author{Janne Junnila}
\address{University of Helsinki, Department of Mathematics and Statistics, P.O. Box 68, FIN-00014 University of Helsinki, Finland}
\email{janne.junnila@helsinki.fi}
\author{Eero Saksman}
\address{Department of Mathematical Sciences, Norwegian University of Science and Technology (NTNU), NO-7491 Trondheim, Norway}
\address{University of Helsinki, Department of Mathematics and Statistics, P.O. Box 68, FIN-00014 University of Helsinki, Finland}
\email{eero.saksman@helsinki.fi}
\author{Christian Webb}
\address{Department of mathematics and systems analysis, Aalto University, P.O. Box 11000, 00076 Aalto, Finland}
\email{christian.webb@aalto.fi}

\begin{document}

\title[Decompositions of log-correlated fields]{Decompositions of log-correlated fields with applications}

\begin{abstract}
In this article we establish novel decompositions of Gaussian fields taking values in suitable spaces of generalized functions, and then use these decompositions to prove results about Gaussian multiplicative chaos.

We prove two decomposition theorems. The first one is a global one and says that if the difference between the covariance kernels of two Gaussian fields, taking values in some Sobolev space, has suitable Sobolev regularity, then these fields differ by a H\"older continuous Gaussian process. Our second decomposition theorem is more specialized and is in the setting of Gaussian fields whose covariance kernel has a logarithmic singularity on the diagonal -- or log-correlated Gaussian fields. The theorem states that any log-correlated Gaussian field $X$ can be decomposed locally into a sum of a H\"older continuous function and an independent almost $\star$-scale  invariant field (a special class of stationary log-correlated fields with 'cone-like' white noise representations). This decomposition holds whenever the term $g$ in the covariance kernel $C_X(x,y)=\log(1/|x-y|)+g(x,y)$ has locally $H^{d+\varepsilon}$ Sobolev smoothness. 

We use these decompositions to extend several results that have been known basically only for $\star$-scale invariant fields to general log-correlated  fields. These include the existence of critical multiplicative chaos, analytic continuation of the subcritical chaos in the  so-called inverse temperature  parameter $\beta$,
 as well as generalised Onsager-type covariance  inequalities which play a role in the study of imaginary multiplicative chaos.
 \end{abstract}

\maketitle

\section{Introduction}\label{sec:intro}

Gaussian multiplicative chaos (GMC) measures are random measures that can be formally thought of as the exponential of a log-correlated Gaussian field. They have connections to many models in mathematical physics such as 2d quantum gravity \cite{DuSh,DKRV,KRV}, SLE \cite{AJKS,DuSh2,Sheffield}, and random matrices \cite{W,BWW,LOS}, as well as to number theory \cite{SW}. A good review of Gaussian multiplicative chaos theory is given in \cite{RhVa}. 

Many of the key results in the theory are proven under rather strong assumptions on the field, which one would not expect to be required. The goal of this article is to partially rectify the situation via  new decomposition methods for the underlying log-correlated field.  Before describing our results in more detail, let us briefly recall  how GMC measures are constructed.

A typical construction of a GMC measure  goes  as follows: given a sequence of continuous approximations $X_n$ of a given log-correlated field $X$, i.e. a Gaussian field with a covariance kernel satisfying $\E X(x) X(y) = -\log |x-y| + O(1)$, one constructs a sequence of approximating measures
\[d\mu_n(x) = e^{\beta X_n(x) - \frac{\beta^2}{2} \E X_n(x)^2} \, dx,\]
where $\beta \in \R$ is a parameter. Then, under fairly general conditions for the approximations $X_n$, for $\beta \in (-\sqrt{2d}, \sqrt{2d})$ the sequence $\mu_n$ converges in probability in the weak$^*$-topology of Radon measures (see Definition \ref{def:weak*probability} for the definition of this concept), and the limiting measure $\mu$ is almost surely non-zero and singular with respect to the Lebesgue measure \cite{Kahane, RhVa,Berestycki,S}. The range $|\beta| < \sqrt{2d}$ of parameter values is called \emph{subcritical}, the case $|\beta| = \sqrt{2d}$ is called \emph{critical} and $|\beta| > \sqrt{2d}$ is \emph{supercritical}. For critical and supercritical $\beta$ the standard normalization scheme yields a limiting measure that is almost surely $0$.

In the critical $\beta = \sqrt{2d}$ case, there are two ways to modify the renormalization to obtain a non-trivial limiting measure. The first way is the so-called Seneta--Heyde normalization, where one looks at the sequence
\[d\mu^{(crit)}_n(x) = \sqrt{\E X_n(x)^2} d\mu_n(x),\]
where the extra factor $\sqrt{\E X_n(x)^2}$ blows up at just the right rate to yield a non-trivial limiting measure. The second approach (yielding the same limiting object  up to a deterministic multiplicative constant) is known as the derivative normalization, where one looks at the derivative of $\mu_n$ with respect to $\beta$ at $\beta=\sqrt{2d}$ and defines
\[{d}\mu'_n(x) = - {\left.\frac{d}{d\beta}\right|_{\beta=\sqrt{2d}}} {d}\mu_n(x) = -(X_n(x) - {\sqrt{2d}}\E X_n(x)^2) {d}\mu_n(x).\]
The existence and uniqueness of critical Gaussian chaos has been studied in various papers \cite{DRSV1,DRSV2,JS,P}. In particular, the existence was established basically only for so called $*$-scale invariant fields, which form a rather specific class of log-correlated fields. 

GMC measures are closely related to another model of random measures called multiplicative cascades. Multiplicative cascades are based on a tree structure, which makes the analysis of these models relatively simple due to the  independence relations that the tree structure induces. Independence properties are also  present to some degree  in certain very specific   approximations of log-correlated Gaussian fields -- in particular in the white noise type approximations of $\star$-scale invariant fields. The special structure of these fields has allowed one to prove in the GMC setting many important results that were already known for   cascades and whose proofs heavily depended on the cascade structure. Such results include the aforementioned existence of critical chaos \cite{DRSV1,DRSV2} and   analyticity of  subcritical chaos in the parameter $\beta$ \cite{AJKS}.

Our first decomposition theorem is a  general result stating that if the difference between the covariance kernels of two (possibly distribution valued) Gaussian fields is regular enough, then the fields can be constructed on the same probability space so  that they differ by a H\"older continuous Gaussian process only. One should note that this is trivial if the difference of the two covariances is a covariance in itself -- in this case the H\"older continuous process could be chosen to be independent of one of the fields, while the general result is less obvious. The precise statement of the theorem is the following -- for definitions of the relevant function spaces, see Section \ref{sec:spaces}.

\medskip

\noindent{\bf Theorem A}\quad {\sl For an exponent $\alpha\in \R$ and  a domain $U\subset \R^d$, let $X_1$ and $X_2$ be two centered Gaussian fields which are random elements in $H_{loc}^\alpha(U)$ with covariance kernels $C_1$ and $C_2$. Let us assume\footnote{The regularity assumption of $C_1$ and $C_2$ can be easily relaxed -- only the regularity of the difference $C_1-C_2$ is important. In fact, the same proof would work if we just assumed $C_1,C_2$ to be in some suitable space of generalized functions, but to avoid unnecessary abstraction, we focus on the case stated in the theorem.} that $C_1,C_2\in L_{loc}^1(U\times U)$ and that for some $\varepsilon>0$ one has $C_1-C_2\in H_{loc}^{d+\varepsilon}(U\times U)$.
	Then, for any bounded subdomain $U'$ with $\overline{U'}\subset U$,  we may construct copies of the fields ${X_1}$ and $X_2$ on a common probability space in such a way that 
	\[
	X_1=X_2+G \quad \textrm{on}\;\; U'
	\]
	for some Gaussian process $G$ which is almost surely H\"older continuous on $U'$.
}

\medskip

\noindent In this paper we use the result only in the setting of log-correlated fields, but we expect that the theorem might  turn out to be useful in other applications as well.
\smallskip

Our second decomposition theorem, which is more specialized in that it applies only to log-correlated fields and is only local in nature, has the benefit of constructing a H\"older continuous Gaussian process which is independent of one of the fields. We will later leverage this independence to prove analyticity of multiplicative chaos in the inverse temperature parameter and suitable Onsager inequalities for logarithmic covariances. The   theorem states that  locally we can write any log-correlated Gaussian field as a sum of a H\"older continuous function and a very special log-correlated field with particular scaling properties. We will refer to this class of special fields as \emph{almost $\star$-scale invariant ones} -- see Remark \ref{rem:0} for discussion about such objects. To avoid technical details, we state a restricted version of the theorem here -- for a more extensive version of the theorem, see Theorem \ref{th:deco} in Section \ref{se:deco}.

\medskip

\noindent{\bf Theorem B}\quad {\sl Assume that $X$ is a log-correlated field on a domain $U\subset\R^d$, whose covariance $C_X(x,y)=\log (1/|x-y|)+g(x,y)$ satisfies $g\in H^{{d+\varepsilon}}_{loc}(U\times U)$ for some $\varepsilon>0$  $($again, see Section {\rm \ref{sec:spaces}} for the definition of the local Sobolev space$)$. Then, any point $x_0 \in U$ has a neighbourhood in which  $X$ can be decomposed {\rm (}possibly on an extended probability space{\rm )} as
	$$
	X\;\;=\;\; L+R,
	$$
	where $L$ is an almost $\star$-scale invariant field $($see Remark {\rm \ref{rem:0}} for the notion$)$, and $ R$ is a regular Gaussian field with H\"older continuous realisations. Moreover,  $L$ and $R$ are independent.} 

\medskip

\noindent Theorem A would of course yield a more global version of such a result, but without the independence assumption. 

Theorem A is proven in Section \ref{se:deco1} and Theorem B is proven in Section~\ref{se:deco}. They have several strong corollaries, that are the the topic of the remaining sections of the paper. Namely,  we deduce the existence of critical chaos  (Theorem~\ref{th:critical} below) and analyticity (Theorem~\ref{th:holomorphic} below)  for a fairly general class of log-correlated fields. These have been known open questions and are interesting since in applications one meets mostly fields that are not of  $\star$-scale invariant type. The general result on analyticity also   implies strong  regularity of the dependence on $\beta$ and provides new  gateways for establishing convergence to chaos for real values of $\beta$. In addition, we prove a general Onsager-type covariance inequality in all dimensions (previously corresponding inequalities have been  proven in the case $g\equiv 0$ in dimension 2), which  is a key tool  for bounding the growth rate of moments of general imaginary Gaussian chaos in \cite{JSW}. We also expect  that the decomposition result might be helpful in studying the fine distribution of the maxima of general log-correlated fields, and also for analogous  extensions of the theory of supercritical chaos -- see Remark \ref{rem:scmax}.

\medskip

{\bf Acknowledgements:} The first author was supported by the the Doctoral Programme in Mathematics and Statistics at the University of Helsinki. The first and second authors were supported by the Academy of Finland CoE \lq Analysis and Dynamics\rq, as well as the Academy of Finland Project \lq Conformal methods in analysis and random geometry\rq. C.W. was supported by the Academy of Finland grant 308123. We wish to thank an anonymous reviewer for carefully reading the article, pointing out misprints, and making suggestions which have improved the quality of the article.

\section{Preliminaries}\label{sec:logcor}

\subsection{Log-correlated fields}\label{subse:log correlated}

A (distribution-valued) centered Gaussian process $X$ on a domain $U\subset \R^d$ with a covariance (kernel)
\begin{equation}\label{eq:cov}
C_X(x,y)=\E X(x)X(y)=\log |x-y|^{-1}+g(x,y),
\end{equation}
where $g\in C(U\times U)$ is  called a log-correlated field. Naturally then $C_X$ is  symmetric and positive semi-definite: $C_X(x,y)=C_X(y,x)$ and
$$
\int C_X(x,y) f(x) \overline{f(y)} \, dx \, dy \ge 0
$$ 
for all $f \in C^\infty_c(\R^d)$. Conversely, given such a covariance kernel, and assuming e.g. that 
\begin{align}\label{eq:assumptions}
\begin{cases} \quad g \in L^1(U \times U)\cap C(U\times U), \quad  g\;\;\textrm{is bounded from above in}\;\; U\times U, \quad\textrm{and}\\
\quad U\subset\R^d\quad \textrm{is a  bounded domain}.\end{cases}
\end{align}
one may easily prove the existence of a Gaussian field with covariance \eqref{eq:cov} and with nice regularity properties like a.s. $X\in H^s({\R^d})$ for any $s<0$, when the field is understood as zero outside of $U$ (see e.g \cite[Proposition 2.3]{JSW}). 
It will be  convenient to extend $C_X(x,y)$ to $\reals^d \times \reals^d$ by setting $C_X(x,y) = 0$ whenever $(x,y) \notin U \times U$. 

The Gaussian multiplicative chaos $``e^{\beta X}"$  is defined by replacing $X$ by suitable approximations $X_n$, which are  a.s. continuous Gaussian fields.  One exponentiates, renormalizes and then removes the smoothing by taking an appropriate limit in $n$. We refer to e.g. the review \cite{RhVa} for basic definitions and properties of multiplicative chaos. Usually the  approximating fields $X_n$ are given in terms of the problem under consideration, or often they are just standard mollifications of $X$. Most of the approximations one encounters have certain useful properties in common, that are described by the notion of a \emph{\lq standard approximation\rq}:

\begin{definition}[Standard approximation]\label{def:standard}
	Let the covariance $C_X$ be as in \eqref{eq:cov} and \eqref{eq:assumptions}. We say that a sequence $(X_n)_{n\geq1}$ of continuous jointly Gaussian centered fields on $U$ is a standard approximation of $X$ if it satisfies:
	
	\begin{itemize}[leftmargin=0.5cm]
		\item[(i)] One has
		$$
		\lim_{(m,n)\to \infty}\E X_m(x)X_n(y)=C_X(x,y),
		$$
		
		\noindent where convergence is in measure with respect to the Lebesgue measure on $U\times U$. 
		
		\item[(ii)] There exists a sequence $(c_n)_{n=1}^\infty$ such that $c_1\geq c_2\geq ...>0$, $\lim_{n\to\infty}c_n=0$,  and for every compact $K\subset U$
		$$
		\sup_{n\geq 1}\sup_{x,y\in K}\left|\E X_n(x)X_n(y)-\log \frac{1}{\max(c_n,|x-y|)}\right|<\infty.
		$$
		
		\item[(iii)] We have 
		$$
		\sup_{n\geq 1}\sup_{x,y\in U}\left[\E X_n(x)X_n(y)-\log\frac{1}{|x-y|}\right]<\infty.
		$$ \enddef
		
	\end{itemize}\end{definition}
	
	\smallskip
	
	A typical standard approximation is obtained by mollifications $ X_{\varepsilon_n}:=\psi_{\varepsilon_n}*X$,
	where $\psi\geq 0$ is a compactly supported smooth  test function with integral 1, $\psi_\varepsilon:=\varepsilon^{-d}\psi(\varepsilon^{-1}\cdot)$  and $\varepsilon_n\searrow 0$ as $n\to\infty$ (see e.g. \cite[Lemma 2.8]{JSW}).

	\subsection{Classical function spaces }\label{sec:spaces}

	We recall here the standard definition of $L^2$-based Sobolev spaces of smoothness index $s\in \R$. One sets
	\begin{equation}\label{eq:sobo}
	H^{s}(\R^d)=\left\lbrace \varphi\in \mathcal{S}'(\R^d): \|\varphi\|_{H^s(\R^d)}^2=\int_{\R^d}(1+|\xi|^2)^s \big|\widehat{\varphi}(\xi)\big|^2 \, d\xi <\infty\right\rbrace,
	\end{equation}
	where $\widehat{\varphi}$ stands for the Fourier transform of the tempered distribution $\varphi$ -- our convention for the Fourier transform is 
	$$
	\widehat{\varphi}(\xi)=\int_{\R^d}e^{-2\pi i \xi \cdot x}\varphi(x)dx
	$$
	for any Schwartz function $\varphi\in \mathcal{S}(\R^d)$. Some basic facts about the spaces $H^s(\R^d)$ 
	are e.g. that they are Hilbert spaces,  for $s>0$, $H^{-s}(\R^d)$ is the dual of $H^s(\R^d)$ with respect to the
	standard dual pairing, $H^s(\R^d)$ is a subspace of $C_0(\R^d)$ for $s>d/2$, i.e. there is a continuous 
	embedding into the space of continuous functions vanishing at infinity, and for $s<-d/2$, compactly 
	supported Borel measures (especially $\delta$-masses) are elements of $H^{s}(\R^d)$.

	Given a domain $U\subset\R^d$, a distribution $\lambda$ belongs to the local Sobolev space $H^s_{loc}(U)$ if for every test function $\varphi\in C_c^\infty(U)$ one has 
	$\varphi\lambda \in H^s(\R^d)$. Moreover, we say that $\lambda_j\to\lambda$ in $H^s_{loc}(U)$ if $\| \varphi(\lambda_j-\lambda)\|_{H^s(\R^d)}\to 0$  as $j\to\infty$ for any $\varphi\in C_c^\infty(U)$.
	
	We shall make use of the standard Sobolev embedding
	\begin{equation}\label{eq:six}
	\| f\|_{L^q(\R^{d})}\leq C'\|  f\|_{H^{s}(\R^{d})} \qquad \textrm{if}\quad \frac{s}{d}=\frac{1}{2}-\frac{1}{q},\quad s<d/2,
	\end{equation}
	and of the super-critical Sobolev embeddings (say for $\delta\in (0,1)$)
	\begin{equation}\label{eq:trois}
	\| f\|_{C^\delta (\R^{d})}\leq C'\|  f\|_{H^{d/2+\delta}(\R^{d})} \qquad \textrm{and}\quad 
	\| f\|_{C^{1+\delta}(\R^{d})}\leq C'\|  f\|_{H^{d/2+1+\delta}(\R^{d})}.
	\end{equation}
	Here for $\delta\in(0,1)$, $C^\delta(\R^d)$ denotes the space of $\delta$-H\"older continuous functions vanishing at infinity and $C^{1+\delta}(\R^d)$ the space of once differentiable functions vanishing at infinity whose derivatives are in $C^\delta(\R^d)$ and both spaces are endowed with their standard norms -- for a proof of the embeddings  and further details, see e.g. \cite[Section 2.8.1]{Tr}.
	
	We also  need a   basic result from  interpolation theory of function spaces: let $s_1, s_2,s'_1,s'_2\in\R$ with $s_1<s_2$ and $s'_1<s'_2$, and assume that the linear operator $T$ (perhaps originally defined only on say $C_c^\infty(\R^d)$) extends both to a bounded operator $T: H^{s_1}(\R^{d})\to H^{s'_1}(\R^{d})$ and to a bounded operator $T: H^{s_2}(\R^{d})\to H^{s'_2}(\R^{d})$. Then for any $\theta\in (0,1)$ the operator $T$ also extends to a bounded operator $T: H^{s}(\R^{d})\to H^{s'}(\R^{d})$, where $s=(1-\theta)s_1+\theta s_2$, $s'=(1-\theta)s'_1+\theta s'_2$ with the norm bound
	\begin{align}\label{eq:interpolation}
	&\| T\; : \; H^{s}(\R^{d})\to H^{s'}(\R^{d})\|\\ \;\leq \; &C\| T: H^{s_1}(\R^{d})\to H^{s'_1}(\R^{d})\|^{1-\theta}\| T: H^{s_2}(\R^{d})\to H^{s'_2}(\R^{d})\|^\theta \nonumber
	\end{align}
	(see \cite[Section 2.4]{Tr}). 
	Moreover, for a fixed function $f$ an application of  H\"older's inequality and definition \eqref{eq:sobo} yields 
	\begin{align}\label{eq:interpolation2}
	\| f \|_{H^{s}(\R^{d})}\;\leq \; \| f \|_{H^{s_1}(\R^{d})}^{1-\theta} \| f \|_{H^{s_2}(\R^{d})}^{\theta}.
	\end{align}
	
	In the proof of the analytic dependence of multiplicative chaos on the inverse temperature parameter it will be convenient to employ Hilbert space valued Hardy spaces. Thus, let $E$ be a separable Hilbert space, $\D$ the open unit disk, and $p\in (1,\infty)$. A  function $f:\D\to E$ belongs to the $E$-valued Hardy space $\mathcal{H}^p(\D,E)$ if $f$ is analytic, i.e. for all $e \in E$ the map $z \mapsto \langle f(z), e \rangle_E$ is analytic, and
	\[\|f\|_{\mathcal{H}^p(\D,E)} := \sup_{0 < r < 1} \Big( \int_0^{2\pi} \|f(r e^{it})\|_E^p \, dt \Big)^{1/p} < \infty.\]
	It is also easy to check by using the Cauchy integral formula that we have the uniform bound
	\begin{align}\label{eq:hardy bound}
	\sup_{|z|\leq 1/2}\|f(z)\|_E\leq C_p\|f\|_{\mathcal{H}^p(\D,E)}.
	\end{align}
	The space $\mathcal{H}^p(\D,E)$, $p \in (1,\infty)$, is reflexive and separable, and hence has the Radon--Nikodym property.
	This can be verified by an elementary argument, or one may deduce it from the general results of \cite{Bl}. One actually notes that $\mathcal{H}^p(\D,E)$ is isometrically isomorphic to  a closed a subspace  $M\subset L^p(\partial \D,E)$ via the Poisson extension. Here  $L^p(\partial \D,E)$ is the standard $E$-valued Lebesgue space, which is separable and reflexive for $p\in (1,\infty)$, and $M$ consists of those elements whose negative Fourier coefficients vanish.
	For general balls  $B\subset\C$, the function $f:B\to E$ belongs to the Hardy space $\mathcal{H}^p(B,E)$ if
	$f\circ \phi\in \mathcal{H}^p(\D,E)$, where $\phi:\D\to B$ is a  bijective affine map, $\phi(z)=az+b$.

	Finally to conclude this section, we record a simple approximation result in Sobolev spaces which is certainly a well-known fact, but for the reader's convenience we provide a proof.  
	\begin{lemma}\label{le:soboapprox}
		Let $K\in H^s(\R^{2d})$ for some $s\geq 0$ be symmetric  $(K(x,y)=K(y,x)$ for $x,y\in\R^d)$ and real-valued. Then for each $\varepsilon>0$, we can find a symmetric real-valued function $K_\varepsilon\in C_c^\infty(\R^{2d})$ such that $\|K-K_{\varepsilon}\|_{H^s}<\varepsilon$ and the integral operator on $L^2(\R^d)$ associated with the kernel $K_{\varepsilon}$ is of finite rank.
	\end{lemma}
	\begin{proof}
		One may obtain a very quick proof by applying a suitable wavelet decomposition of the given function (or  by working with Fourier series). However, we give here an argument that utilises  the very definition of the Sobolev norm.
		Let us begin by defining a function $K^{(R)}$ for $R>1$ by 
		\[
		K^{(R)}(x)=\int_{|\xi|\leq R}e^{2\pi i \xi \cdot x}\widehat K(\xi)d\xi= \int_{|\xi|\leq R}\cos(\xi x)\widehat K(\xi)d\xi ,
		\]
		by the symmetry of  $K$.  We also have that $K^{(R)}$ is real  since  $\widehat K$ is real.
		Obviously  $K^{(R)}\to K$ in $H^s(\R^{2d})$ as $R\to \infty$. Fix $  \varphi_0\in C_c^\infty(\R^d)$ non-negative with $\varphi_0(0)=1$,  define $\varphi_\delta \in C_c^\infty(\R^{2d})$ by setting $\varphi_\delta(x_1,x_2)=\varphi_0(\delta x_1)\varphi_0(\delta x_2)$ for $x_1,x_2\in\R^{d}$, and set $K^{(R,\delta)}:=\varphi_\delta K^{(R)}$. Then it is a classical fact (e.g. easily checked by using the density of smooth functions in $H^{s}$) that $K^{(R,\delta)}\to K^{(R)}$ in $H^s$ as $\delta\to 0$. In order to produce  a finite rank approximation we observe that  $\xi\mapsto g_\xi:=\varphi_\delta\cos(\xi \cdot)$ is a continuous map $\R^{2d}\to H^s(\R^{2d})$. If $s\in\N$, this follows easily by differentiating, and it thus holds for all $s$. The continuity of $\xi\mapsto g_\xi$ allows us to  approximate in a standard manner the integral  representation 
		$$
		K^{(R,\delta)}=\int_{|\xi|\leq R}\widehat K(\xi)g_\xi
		$$
		by a discrete sum $\sum_{n=1}^Na_ng_{\xi_n}$ with judiciously chosen $N$, $a_n\in\R$ and $\xi_n\in B(0,R)$. Finally,  symmetrisation produces the desired approximation kernel.
	\end{proof}

	\medskip
	
	We now turn to our first decomposition result.

	\section{Global decomposition of Gaussian fields: Proof of Theorem A}
	\label{se:deco1}
	
	We start by recalling some basic facts on covariances and integral operators    needed for our purposes here. Assume that $K\in L^2(\R^{2d})$ is real valued and symmetric. Then $K$ is the kernel of a self-adjoint Hilbert-Schmidt operator $T=T_K$ on $L^2(\R^d)$, and we also denote $K=K_T$. In particular, $T$ is compact and  basic spectral theory yields (due the fact that $K$ is real) the existence of orthogonal  and real-valued eigenfunctions $g_1,g_2,\ldots\in L^2(\R^d)$ and real eigenvalues $(\lambda_k)_{k\geq 1}$ such that the (square of the) Hilbert-Schmidt norm $\|T\|^2_{HS}:=\int_{\R^{2d}}K^2(x,y)dxdy=\sum_{k=0}^\infty  |\lambda_k|^2$ is finite. We may then write
	$$
	K(x,y)=\sum_{k=1}^\infty \lambda_kg_k(x)g_k(y).
	$$ 
	with convergence in $L^2(\R^{2d})$, or equivalently, $T=\sum_{k=1}^\infty \lambda_kg_k\otimes g_k$. Conversely, each such sum with real-valued orthonormal $g_k$:s and real-valued square-summable $(\lambda_k)_{k\geq 1}$ defines a Hilbert-Schmidt operator with a real and symmetric kernel. The kernel $K$ is called positive if $\int_{\R^{2d}}K(x,y)\varphi(x)\varphi(y)\geq 0$ for all (real-valued) test functions $\varphi\in C_c^\infty (\R^d)$, which is equivalent to  positivity $T\geq 0$ in the standard operator sense, again since $K$ is real and symmetric. 
	
	The absolute value of the operator $T$,  denoted by $|T|$,  is  the  operator with  kernel
	\begin{equation}\label{eq:absT}
	K_{|T|}(x,y):=\sum_{k=1}^\infty |\lambda_k|g_k(x)g_k(y).
	\end{equation}
	Actually, $|T|$ is the unique bounded operator such that $|T|\geq 0$ and $|T|^2=T^2.$
	The non-linear operation (sometimes called operator Hilbert transform) $T\mapsto |T|$ obviously satisfies
	$\| \, |T|\,\|_{HS}=\|T\|_{HS}$, or equivalently, it keeps the $L^2$-norm of the kernel invariant. By definition the operators
	$T^{\pm} :=(|T|\pm T)/2$ are  positive  and Hilbert-Schmidt.  The proof of Theorem A will be based on use of the decomposition $T=T^+-T^-$ in combination with  an auxiliary result stating that not just the $L^2$-norm, but also the smoothness property $K_T\in H^{s}(\R^{2d})$, $s>0$, remains intact under the operation $T\mapsto |T|:$
	\begin{lemma}\label{le:absreg}
		Let  $K$ be a symmetric and square-integrable kernel on $\R^d$, and denote by $T$ the corresponding operator on $L^2(\R^d)$. Assume that in addition $K\in H^s(\R^d\times \R^d)$ with $s>0$ Then also $K_{|T|}\in H^s(\R^d\times \R^d)$.
	\end{lemma}
	\newcommand{\four}{\mathcal F}
	\begin{proof}
		We use the identity $|T|^2=T^2$ to relate the Sobolev norms of the kernels $K=K_T$ and $K_{|T|}$. Let us first assume that $K$  is smooth with compact support and $T$ is finite rank -- then the same holds for $K_{|T|}$ and $|T|$ according to \eqref{eq:absT} and by the very definition of eigenfunctions.  We denote by $\four_i$ the Fourier transform with respect to the variables $x_{d(i-1)+1},\ldots x_{d(i-1)+d}$, $i=1,2$, so that $\widehat K=\four_1\four_2K$, and we write $e_u(x)=e^{2\pi i u\cdot x}$ for any $u\in\R^d.$ Our assumptions on $K$ and $K_{|T|}$ imply that we may safely compute for any fixed $\xi\in\R^d$ 
		\begin{align*}
		\widehat{\big(T^2e_\xi\big)}(\xi)&=\int_{\R^d}e^{-2\pi i \xi \cdot x}\int_{\R^d\times \R^d}K(x,y)K(y,v)e^{2\pi i \xi\cdot v}dydvdx\\
		&=\int_{\R^d}(\four_1 K)(\xi,y)(\four_2 K)(y,-\xi)dy.
		\end{align*}
		Due to symmetry of $K$ we actually have $(\four_2 K)(y,-\xi)=\overline{(\four_1 K)(\xi,y)}$, whence
		\begin{equation}\label{eq:helppi3}
		\widehat{\big(T^2e_\xi\big)}(\xi)= \int_{\R^d}|(\four_1 K)(\xi,y)|^2dy=\int_{\R^d}|\widehat K(\xi,\eta)|^2d\eta,
		\end{equation}
		where the last equality follows from Parseval's theorem with respect to the variable $y$. Since $|T|^2=T^2$, performing the same computation for $K_{|T|}$ yields the equality
		$$
		\int_{\R^d}|\widehat K_{|T|}(\xi,\eta)|^2d\eta\; =\; \int_{\R^d}|\widehat K(\xi,\eta)|^2d\eta.
		$$

		By symmetry we have $\widehat{K}(\xi,\eta)=\widehat{K}(\eta,\xi)$, and similarly for $K_{|T|}$, so that 
		\[
		\int_{\R^d}|\widehat{K}(\xi,\eta)|^2d\xi=\int_{\R^d}|\widehat{K}_{|T|}(\xi,\eta)|^2d\xi  \quad \text{and} \quad \int_{\R^d}|\widehat{K}(\xi,\eta)|^2d\eta =\int_{\R^d}|\widehat{K}_{|T|}(\xi,\eta)|^2d\eta.
		\]
		This yields for any $s>0$
		\begin{equation}\label{eq:jelppi1}
		\int_{\R^d\times \R^d}|\widehat{K}(\xi,\eta)|^2(1+|\xi|^{2s}+|\eta|^{2s})d\xi d\eta=\int_{\R^d\times \R^d}|\widehat{K}_{|T|}(\xi,\eta)|^2(1+|\xi|^{2s}+|\eta|^{2s})d\xi d\eta.
		\end{equation}
		Since for $s\geq 0$ it holds that $(1+|\xi|^{2s}+|\eta|^{2s})\asymp (1+|\xi|^2+|\eta|^2)^s$  we see that the knowledge $K\in H^s(\R^d\times \R^d)$ implies that also $K_{|T|}\in H^s(\R^d\times \R^d)$. 
		
		In order to deal with the general case, we use Lemma \ref{le:soboapprox} to approximate a general $K\in H^s(\R^{2d})$ by a sequence of kernels $K_n$ which are smooth, have compact support, and the associated operators $T_{K_n}$ have finite rank. In particular, Lemma \ref{le:soboapprox} implies that $\|T_{K_n}-T\|_{\mathrm{HS}}=\|K-K_n\|_{L^2(\R^{2d})}\to 0$ as $n\to \infty$. We next make use of the fact fact that  $\| T_{K_n}-T_{K}\|_{HS}\to 0 $  implies $\| |T_{K_n}|-|T_{K}|\|_{HS}\to 0 $  according to \cite{D}
		(see also \cite{PS}, especially the discussion of the special case $f(t)=|t|$ on p. 376). Equivalently, $K_{|T_n|}\to K_{|T|}$ in $L^2(\R^{2d})$.
		
		We still need to argue that this implies  that $K_{|T|}\in H^s(\R^{2d})$. For this, we note that since $K_{T_n}\to K_{T}\in H^s(\R^{2d})$ and the first part of the proof of this lemma implies that $\|K_{|T_n|}\|_{H^s(\R^{2d})}\asymp \|K_{T_n}\|_{H^s(\R^{2d})}$, the sequence $\|K_{|T_n|}\|_{H^s(\R^{2d})}$ is bounded. Now by a standard Banach-Alaoglu argument, one can extract from $K_{|T_n|}$ a weakly convergent subsequence, with limit say $\widetilde K\in H^s(\R^{2d})$. But since we already know that $K_{|T_n|}\to K_{|T|}$ in $L^2(\R^{2d})$, a standard argument shows that we must have $\widetilde K=K_{|T|}$, which implies that $K_{|T|}\in H^s(\R^{2d})$ and concludes our proof.
	\end{proof}

	The rest of the proof will be divided into separate lemmas.

	\begin{lemma}\label{le:absreg2} 
		Under the assumptions of Theorem A there exist almost surely  H\"older continuous and centered Gaussian processes $G_\pm$ on $\R^d$ such that their covariances satisfy 
		$$
		C_{G_+}(x,y)- C_{G_-}(x,y)= C_1(x,y)-C_2(x,y) \quad \textrm{for}\;\; x,y\in U'.
		$$
	\end{lemma}
	
	\begin{proof} Let $\psi_0\in C^\infty_c(U)$ satisfy $\psi_0\equiv 1$ in a neighbourhood of $U'$, and note that  under the assumptions of Theorem A, the kernel $K(x,y):=\psi_0(x)\psi_0(y)\big(C_1(x,y)-C_2(x,y)\big)$ is an  element of $H^{d+\varepsilon}(\R^d\times \R^d)$. Denote by $T=T_K$ the operator on $L^2(\R^d)$ with kernel $K$. Lemma \ref {le:absreg}  then implies that the kernels $K_{\pm}$ of the positive and symmetric operators $T^\pm:=(|T|\pm T)/2$  also belong to the space $H^{d+\varepsilon}(\R^d\times \R^d)$. Due to the Sobolev embedding \eqref{eq:trois} they thus are H\"older continuous, and by standard regularity theory of Gaussian processes, see e.g. \cite[Theorem 1.3.5]{AT}, they are covariances of  centered Gaussian process $G_\pm$ on $\R^d$ with almost surely H\"older continuous realizations. The statement follows since we have $C_{G_+}- C_{G_-}= C_1-C_2$ on $U'\times U'$.
	\end{proof}

	\begin{lemma}\label{le:gconstr1} Let $A_k,A'_k$, $k\geq 1$ be jointly Gaussian centered variables. In a similar way, assume that  $B_k,B'_k$, $k\geq 1$  are jointly  Gaussian centred variables, but possibly defined on a different probability space, and assume that there is the equality of distributions  
		$$
		(A_k+A'_k)_{k\geq 1} \sim (B_k+B'_k)_{k\geq 1}.
		$$
		Then we may realize all the variables $A_k,A'_k,B_k,B'_k,$ $k\geq 1$ as jointly Gaussian variables on a common probability space in such a way that the distributions  of the double sequences $(A_k,A'_k)_{k\geq 1}$ and $(B_k,B'_k)_{k\geq 1}$ remain intact,  and at the same time there is the almost sure equality $A_k+A'_k=B_k+B'_k$ for all $k\geq 1.$
	\end{lemma}
	\begin{proof}
		Let $G=G_0\oplus G_0^\perp$ be a Gaussian (centred) Hilbert space, where both of  the mutually orthogonal subspaces $G_0$ and $G_0^\perp$ are infinite-dimensional. In turn denote by $M$ the Gaussian Hilbert space obtained as the closed $L^2$-span of all the variables $A_k,A'_k$, $k\geq 1.$ Let $M_0\subset M$ stand for the closed linear span of $A_k+A'_k$, $k\geq 1$, in $M$ so that $M=M_0+M_0^\perp$. Let $\Psi_A:M\to G$ be an isometric imbedding (not necessarily surjective) such that $\Psi_A(M_0)\subset G_0$ and $\Psi_A(M_0^\perp)\subset G_0^\perp.$ We may then pick an analogous isometric embedding $\Psi_B$ from the linear span of all the variables $B_k,B'_k$, $k\geq 1$, into $G$  such that $\Psi_A(A_k+A'_k)=\Psi_B(B_k+B'_k)$ for all $k\geq 1.$ Then the variables $\Psi_A (A_k),\Psi_A(A'_k), \Psi_B(B_k)$, and $\Psi_B(B'_k)$ have the desired properties.
	\end{proof}
	\begin{lemma}\label{le:gconstr}
		Under the assumptions of Theorem A and in the notation of Lemma \ref{le:absreg2}, we may construct jointly Gaussian copies of the fields $X_1, X_2, G_\pm$ on $U'$ on a common probability space so that almost surely
		\[
		X_1+G_-= X_2+G_+ \quad\textrm{on}\;\; U'.
		\]
	\end{lemma}
	\begin{proof}
		By construction, just by computing covariances, $\psi_0X_1+G_-$ (where $G_-$ is independent of $X_1$) has the same distribution as $\psi_0X_2+G_+$ (where $G_+$ is independent of $X_2$).  Here $\psi_0$ is from the proof of Lemma \ref{le:absreg2}, and we view both of these sums as $H^{\min{(\alpha,0)}}(\R^d)$-valued random variables. Due to the compact support of $\psi_0$, we may pick a cube $Q_0=(-a,a]^{d}\subset\R^{d}$ so large that the support of $\psi_0$, and hence of all the fields we are considering here are contained already in the cube $\frac{1}{2}Q_0= (-a/2,a/2]^{d}.$ Develop the Gaussian field $ \psi_0 X_1$ into Fourier series
		$$
		\psi_0X_1 =\sum_{k\in \Z^d}A_{k}e_{k},
		$$
		where $(e_{k})_{k\in\Z^d}$ stands for the standard Fourier basis in the cube $Q_0$. In a similar vein,
		let  $(A'_{k})_{k\in\Z^d}$, $(B_{k})_{k\in\Z^d}$, and $(B'_{k})_{k\in\Z^d}$, respectively, stand for the Fourier coefficients of $G_-$, $\psi_0 X_2$, and $G_+$, respectively.  
		The claim is now an easy consequence of Lemma \ref{le:gconstr1}.
	\end{proof}
	
	\begin{proof}[Proof of Theorem A] The statement  follows immediately from  Lemma \ref{le:gconstr} by choosing $G:=G_+-G_-$ on $U'$. \end{proof}
	
	\begin{remark}\label{rem:gl_opt}
		As suggested to us by an anonymous referee, we now discuss the optimality of Theorem A, namely we discuss what happens  if we only assume that $C_1-C_2\in H_{loc}^{d}(\R^{2d})$. More precisely, we will provide an example of such a situation where the conclusion of Theorem A does not hold. Our example, which is in the case $d=1$, is built from the kernel 
		\[
		K(x)=\varphi(x)\log\log \frac{1}{|x|},
		\]
		where $\varphi:\R^2\to [0,\infty)$ is a smooth symmetric function which is non-zero at the origin and whose support is contained in the open unit ball of $\R^2$. One can then readily check that the gradient of $K$ is square integrable, so that $K\in H^{1}(\R^2)$. 
		
		Now $K=K(x,y)$ is symmetric, so the positive definite kernels $K_\pm$ -- that is the kernels of the operators $T_{|K|}\pm T_{K}$ -- are covariance kernels. Moreover, by Lemma \ref{le:absreg}, $K_\pm \in H^1(\R^2)$, and they of course satisfy $K_+-K_-=K$. Thus, since $K$ is unbounded, at least one of the kernels $K_+,K_-$ is also unbounded near the origin. Let $C_1$ be (one of) the unbounded one(s) and choose $C_2\equiv 0$. 
		
		Let $X_1$ be a centered Gaussian field with covariance $C_1$ and set $X_2\equiv0$. If $X_1$ were almost surely H\"older continuous, then $C_1$ would be continuous (after possibly modifying it in a set of Lebesgue measure zero).  This contradicts the unboundedness of $C_1$ at zero and shows that Theorem A does not hold in case we simply assume that $C_1-C_2\in H^{d}_{loc}(\R^{2d})$.
		
		It is an open question what happens if one replaces in Theorem A the scale $H^s$ by some other scale, for example the Besov spaces $B_{p,q}^s$, which also cover the H\"older-spaces.
		\hfill $\diamond$ 
	\end{remark}
	
	We now move on to proving our local decomposition.

	\section{Local decomposition of log-correlated fields: Proof of Theorem B}\label{se:deco}
	
	In this section we establish our basic result on local splitting of log correlated fields. We start by constructing a suitable split of $\star$-scale invariant fields. 
	
	\subsection{Decomposing $\star$-scale invariant fields}\label{subse:decomposing}
	
	Roughly speaking, a $\star$-scale invariant log-correlated Gaussian field { $Y$} has a translation invariant covariance { $C_Y$} of the form { $C_Y(x,y)=K(x-y)$} with the representation
	\begin{equation}\label{eq:logno}
	K(x):=\int_0^\infty k(e^{u}x)du,\quad x\not=0
	\end{equation}
	where $k(x-y)$ is a covariance  on $\R^d$, with some regularity, { $k(0) = 1$}, and $k$ has suitable decay at infinity, e.g. 
	\begin{equation}\label{eq:logno2}
	|k(x)|\lesssim  |x|^{-a}\qquad \textrm{for some}\quad a>0,
	\end{equation}
	which in particular makes $K$ well-defined. We call $k$ the `seed covariance function' of the construction.
	
	These fields are quite natural because  \eqref{eq:logno} implies that they possess a certain self-similar structure, and they appear  in the characterisation of $\star$-scale invariant measures \cite{ARV}.  Moreover, several basic results  related to multiplicative chaos have been previously established only in connection of these fields, including the construction of critical and supercritical chaos,  analytic continuation in the inverse temperature parameter, and sharp estimates for maxima  of log-correlated Gaussian fields \cite{AJKS,DRSV1,DRSV2,Madaule,MRV}. In the following auxiliary result we revisit the construction of such fields and, in particular, introduce a useful split of the constructed field { $Y=L+S$} into independent summands where $L$ is an `almost $\star$-scale invariant' field (see Remark \ref{rem:0} below for clarification of  our terminology here) and $S$ has H\"older continuous realisations. For our later purposes non-rotationally invariant $\star$-scale invariant covariance kernels  are not so useful, so in the what follows we will always assume rotational symmetry\footnote{Thus, the covariance kernel { $C_Y$} satisfies { $C_Y(x,y)=K(x-y)$}, where $K:\R^d\to\R$ is rotation symmetric.}.
	
	\begin{proposition}\label{pr:split} Fix $\delta\in (0,1)$ and assume that $k:\R^d\to\R$ is a rotation invariant and $\varepsilon$-H\"older-continuous function with some $\varepsilon>0$. Moreover, let $k$ be such that $(x,y)\mapsto k(x-y)$ yields a translation invariant covariance on $\R^d$. We also assume that $k(0)=1$ and $k$ satisfies the decay \eqref{eq:logno2}.  Then there is a constant $\varepsilon_0 > 0$ and  almost surely continuous centered {\rm (}jointly{\rm )} Gaussian fields {$(x,t)\mapsto L_t(x)$} and $(x,t)\mapsto S_t(x)$, indexed by $(x,t)\in \R^d\times (0,\infty)$   with the following properties:
		
		\smallskip
		
		\noindent {\rm (i)}\quad The fields { $L_{\{\mathbf{\cdot}\}}(\mathbf{\cdot})$} and 
		$S_{\{\mathbf{\cdot}\}}(\mathbf{\cdot})$ are independent of each other. In addition,
		for any $0\leq t_0<t_1<\ldots <t_n,$ the increment fields {$L_{t_j}-L_{t_{j-1}}$} on $\R^d$, $j=1,\ldots,n$, are independent, and the same is true for the increments of $S_{\{\mathbf{\cdot}\}}(\mathbf{\cdot})$. 
		For arbitrary  $x,x'\in\R^d$ and $t,t'\in(0,\infty)$ we have 
		\begin{align}\label{eq:covLR} 
		\begin{array}{lll}\E { L_t(x)L_{t'}(x')} &=& \int_0^{t\wedge t'}k\big(e^{u}(x-x')\big)(1-e^{-\delta u})du\qquad \textrm{and}\\ &&\\
		\E S_t(x)S_{t'}(x')&=& \int_0^{t\wedge t'}k\big(e^{u}(x-x')\big)e^{-\delta u}du \end{array}
		\end{align}
		\smallskip
		
		\noindent {\rm (ii)}\quad Almost surely the realizations {$(x,t)\mapsto L_t(x)$}
		and $(x,t)\mapsto S_t(x)$ are $\varepsilon_0$-H\"older continuous on any compact subset of $\R^d\times (0,\infty)$.
		
		\smallskip
		
		\noindent {\rm (iii)}\quad The field $S_{\{\mathbf{\cdot}\}}(\mathbf{\cdot})$ extends continuously to $\R^d\times [0,\infty]$. Moreover, the field $S:=S_\infty$ is almost surely  $\varepsilon_0$-H\"older continuous, and we have almost surely
		$$
		\|S_t-S\|_{C^{\varepsilon_0}(B)}\to 0\quad \textrm{as}\quad t\to\infty
		$$
		for any closed ball $B\subset\R^d$. The covariance $C_S$ is  H\"older continuous.
		
		\smallskip
		
		\noindent {\rm (iv)}\quad As $t\to\infty$, it holds that  $L_t \to L$ {almost surely}, where { $L$} is a log-correlated field on $\R^d$, and the convergence is in $H^{s}_{loc}(\R^d)$ for any $s<0$. The fields {$(L_t)_{t\geq 0}$} provide a standard approximation of $L$ in the sense of Definition \ref{def:standard}.
		
		\smallskip
		
		\noindent {\rm (v)}\quad For $t>0$ set {$Y_t:=L_t+S_t$}, and denote {$Y:=L+S$}. The field {$Y$} is a $\star$-scale invariant log-correlated field whose covariance kernel is obtained from \eqref{eq:logno}, and {$Y_t\to Y$} {almost surely} in $H^{s}_{loc}(\R^d)$ for any $s<0$ as $t\to\infty$. Moreover, the fields {$(Y_t)_{t\geq 0}$} provide a standard approximation of {$Y$}. 
		
		\smallskip
		
		\noindent {\rm (vi)}\quad Assume in addition that $k\in H^{s}_{loc}(\R^d)
		$ for some $s>d$. Then 
		the covariance {$C_Y$} satisfies 
		{
			$$
			C_Y(x,y)=\log(1/|x-y|)+ g(x,y),
			$$ 
		}
		where $g\in H^{s'}_{loc}(\R^{2d})$ with some $s'>d.$ 
		
	\end{proposition}
	\begin{proof}
		We first verify that the expressions on the right hand side of \eqref{eq:covLR} are covariances on $\R^d\times (0,\infty)$.  In the case of $S$ we need to prove for any $x_1,\ldots, x_n\in\R^d$, any $ t_1, \ldots ,t_n\in (0,\infty)$, and arbitrary reals $a_1,\ldots a_n$  that one has
		$$
		\sum_{j,k=1}^na_ka_j\int_0^{t_k\wedge t_j}k(e^{u}(x_j-x_k))e^{-\delta u}du\geq 0.
		$$
		By symmetry we may assume that $0\leq t_1\leq t_2\leq \ldots \leq t_n<\infty$. Set $t_0=0$.
		The desired positivity is then seen directly from the covariance property of $k$ by writing the left hand side in the form
		$$
		\sum_{\ell=1}^n\int_{t_{\ell-1}}^{t_\ell}\Big(\sum_{j,k=\ell}^n a_ka_jk(e^{u}(x_j-x_k))\Big)e^{-\delta u}du.
		$$
		Moreover, the covariance is clearly locally $\varepsilon$-H\"older with respect to the space-variables and even locally Lipschitz with respect to the $t$-variables. In particular, it is jointly $\varepsilon$-H\"older over compacts, and the existence of an almost surely  locally  $\varepsilon_0$-H\"older field $S_{\{\mathbf{\cdot}\}}(\mathbf{\cdot})$  follows from the standard regularity theory of Gaussian fields, see e.g. \cite[Theorem 1.3.5]{AT}. The proof for the field $L_{\{\mathbf{\cdot}\}}(\mathbf{\cdot})$  is similar, and we may naturally construct both fields on a common probability space so that they will be independent of each other. This yields (i) and (ii) as the stated independence of increments  follows from the covariance structure.
		
		In order to verify (iii) it is convenient to  reparametrize by considering $\widetilde S_t(x):=S_{\log(1/t)}(x)$, $t\in (0,1]$.  Then the field $\widetilde S $ has the covariance structure
		\begin{equation}\label{eq:covR}
		\E \widetilde S_t(x)\widetilde S_{t'}(x')= \int^1_{t\vee t'}k\big(u^{-1}(x-{x'})\big)u^{\delta-1}du, 
		\end{equation}
		which yields an extension to a covariance on $\R^d\times [0,1]$. In order to estimate the continuity of the extended covariance, we set $\varepsilon'=1/2\min(\varepsilon, \delta).$  We clearly have $\delta$-H\"older continuity with respect to the $t$-variables, and given $|x-x_0|\leq r\leq1$ and $|y-y_0|\leq r\leq 1$ one obtains
		\begin{align*}
		|C_{\widetilde S}(x,y,t,t')-C_{\widetilde S}(x_0,y_0,t,t')|\; \leq& \; \int^1_0\big| k\big(u^{-1}(x-y)\big)-k\big(u^{-1}(x_0-y_0)\big)\big| u^{\delta-1}du\\
		\lesssim&\; \int^1_0(2r/u)^{\varepsilon'} u^{\delta-1}du \;\lesssim r^{\varepsilon'}.
		\end{align*}
		Thus $\widetilde S$ has a  $\varepsilon'$-H\"older continuous covariance, which again implies existence of a modification that is a.s. H\"older continuous on compacts, and we may assume that we obtained our initial construction of $S$ from such a $\widetilde S$. This  and decreasing the $\varepsilon_0$ obtained in (ii) if needed clearly yields (iii).
		
		Towards (iv) and (v),  we first  note that since {$Y_t=L_t+S_t$}, with independence of the summands, the uniform convergence properties of the field  $S_t$ and its covariance as $t\to\infty$ established in (iii) show that it is enough to treat {}$Y_t$. First of all, {$C_{Y_t}(x,y)=K_t(x-y)$}, where 
		$$
		K_t(x):=\int_0^tk(e^ux)du.
		$$
		We shall perform comparison with the special case\footnote{It is well known \cite[Theorem~1.14]{SteinWeiss} that $\widehat{k^{(0)}}(z)=\frac{c_d}{2} e^{-2\pi |z|}$, where $c_d > 0$ is the area of the unit $d$-sphere $S^d \subset \reals^{d+1}$. As $\widehat{k^{(0)}}(z) > 0$, $k^{(0)}$ gives rise to a translation invariant covariance.} $k= k^{(0)}$ where $ k^{(0)}(x):=(1+|x|^{2})^{-(d+1)/2}$.  Then it follows that
		$$
		K^{(0)}_t(x)= \int_{|x|}^{|x|e^t}\frac{dy}{y(1+y^2)^{(d+1)/2}} = F(|x|)-F(|x|e^t),
		$$
		where $F(y):=\int_y^\infty u^{-1}(1+u^2)^{-(d+1)/2}$ is positive, continuous and decreasing on $(0,\infty)$, and it  obviously satisfies $\lim_{y\to\infty} F(y)= 0$ and $\lim_{y\to 0^+} F(y)-\log(1/y)=0.$ Actually, we note that $F(y)-\log(1/y)$ has smooth continuation over $0$. From this one deduces that the covariances $K^{(0)}_t(x)$ satisfy the conditions of Definition \ref{def:standard}, and 
		together with  the martingale property of {$(Y^{(0)}_t)_{t\geq 0}$}  we deduce that the fields {}$(Y^{(0)}_{t_n})_{n\geq 1}$ yield a standard approximation   to the limit field {$Y^{(0)}$} with covariance kernel $C_{Y^{(0)}}(x,y)= F(| x-y|)$.  In order to treat the general case, we denote by $e_1$ the first unit vector and note that for any $k$ satisfying the conditions of the theorem, one has ${K}^{(0)}_t(0)={K}_t(0)$  and
		$$
		K(x)-K_t(x)={K}^{(0)}(x)-{K}^{(0)}_t(x)+r(e^t|x|), \qquad \textrm{for}\quad x\not=0
		$$
		where $r(y):=\int_{y}^\infty u^{-1}(k- k^{(0)})(ue_1)du$.  The function  $r$ is continuous and uniformly bounded on $[0,\infty)$ with $\lim_{y\to\infty}r(y)=0.$ To check these claims, one applies the H\"older continuity of $ k^{(0)}-k$, the fact that $ k^{(0)}(0)=k(0)=1$, and  the estimate \eqref{eq:logno2}. Since  we know the desired claim (namely being a standard approximation) for the fields {$Y^{(0)}_t,$}  the above equality implies it for the fields {}$Y_t$.  
		
		In order to prove the stated convergence in $H^s_{loc} (\R^d)$, we fix $s\in (-1/2,0)$ and pick a test function $\phi\in C^\infty_c(\R^d)$.  By definition and (ii), $t\mapsto \phi {Y_t}$ is a martingale that takes values in $C_c(\R^d)$, hence in $H^s(\R^d)$, whence (as the latter space is a separable Hilbert-space, see e.g. \cite[Theorem 3.61, Theorem 1.95]{HNVW}) the convergence follows as soon as we prove the $L^2(H^s(\R^d))$-boundedness of the martingale. By the proof of \cite[Proposition 2.3]{JSW}) we have for any fixed $s<0$ and all $t\geq 0$
		\begin{align*}
		\E \|\phi {Y_t}\|_{H^{s}(\reals^d)}^2 & \leq  c_s\int_{\R^d\times\R^d} \frac{\phi(x)\phi(y){C_{Y_t}}(x, y)}{|x - y|^{d +2s}} \, dx \, dy \; \\
		&\lesssim\int_{({\rm supp\; }\phi)^2} \frac{1+\log^+(|x-y|^{-1})}{|x - y|^{d +2s}} \, dx \, dy <\infty,
		\end{align*}
		which is the desired uniform bound. It is then clear that the limit field {$Y$} has the stated covariance structure \eqref{eq:logno}. The facts that {$L$} is log-correlated and {$Y$} is $\star$-scale invariant follow immediately from the covariance structures and the relevant definitions. This concludes the proof of (iv) and (v).
		
		Finally, we turn to the statement (vi) and examine the local smoothness of $g$ in the decomposition \eqref{eq:cov} for the field {$Y$}. Recall that we are now assuming that $k\in H_{loc}^s(\R^d)$ for some $s>d$. Let us assume first in addition that $k$ has compact support, whence $K$ also has, and then the Fourier transform $\widehat K$ is smooth.
		It is readily checked that $h(\cdot-\cdot)\in H^{s}_{loc}(\R^{2d})$ if $h(\cdot)\in H^{s}_{loc}(\R^{d})$ with $s\geq 0$, which means that, as everything is translation invariant, to get a hold of the regularity of $g(x,y)$ we only need to study the local smoothness of $K-\log(1/|\cdot|)$ as a function on $\R^d$.
		We may write for $\xi\not=0$
		\begin{align*}
		\widehat K(\xi)&= \int_0^1\widehat k(u\xi)u^{d-1}du= |\xi|^{-d} \int_0^{|\xi|}\widehat k(ue_1)u^{d-1}du\\ &=
		|S_{n-1}|^{-1}|\xi|^{-d} - |S_{n-1}|^{-1}|\xi|^{-d} \int_{|z|>|\xi|}\widehat k(z)dz,
		\end{align*}
		where  $e_1$ is the first unit vector,  $|S_{n-1}|$  stands for the $(n-1)$-measure of the unit $(n-1)$-sphere,  and we used that $k(0)=1=\int_{\R^d}\widehat k(z)dz.$ 
		Since outside the origin the Fourier-transform of $\log(1/|\cdot|)$ is given by the function $|S_{n-1}|^{-1}|\xi|^{-d}$, the above equality enables us to write
		$$
		K=\log(1/|\cdot|) +h+h_0,
		$$
		where the Fourier transform of $h_0\in \mathcal{S}'(\R^d)$  is supported in $\overline{B(0,1)}$,  so that $h_0$ is smooth, and
		$$
		\widehat h(\xi):=  \begin{cases}-|S_{n-1}|^{-1}|\xi|^{-d} \int_{|z|>|\xi|}\widehat k(z)dz, &\textrm{for}\;\; |\xi|>1,\\
		0 &\textrm{for}\;\; |\xi|\leq 1.
		\end{cases}
		$$
		It is then enough to check that  $\|h\|_{H^{s'}(\R^d)}<\infty$ for some $s'>d$, which in turn follows as soon as we verify that $|\widehat h(\xi)|\leq |\xi|^{-a}$ for some  $a>3d/2$. By Cauchy--Schwarz 
		\begin{align*}
		|\xi|^d|\widehat h(\xi)|\; &\leq\;  \Big(\int_{|z|>|\xi|}|\widehat k(z )|^2(1+|z|^2)^{s}dz\Big)^{1/2} \Big(\int_{|z|>|\xi|}(1+|z|^2)^{-s}dz\Big)^{1/2} \\
		&\leq C|\xi|^{d/2-s},
		\end{align*}
		which yields the claim (recall here that we are assuming that $s>d$, $k$ is compactly supported, and that $k\in H_{loc}^s(\R^d)$, which in the compactly supported case is of course equivalent to $k\in H^s(\R^d)$). To finish, our computations in this part did not use the fact that $k$ is a covariance. By writing 
		$K(x)=\int_{|x|}^{\infty} k(ue_1)u^{-1}du$ we see that the smoothness of $K-\log(1/|\cdot|)$ on a given ball $B(0,R)$ depends only on (say) values of $k$ on $B(0,2R)$, and  we infer that the general situation reduces to the case of compact $k$.
	\end{proof}
	\begin{remark}\label{rem:0} In case $k$ is compactly supported, the covariance structure \eqref{eq:covLR}  of the field {$L$} shows that the field admits a representation in terms of a  weighted hyperbolic white noise on $\R^d\times (0,1)$ such that formally {}$L(x)$ is given by integrating the white noise against the function $h(x-\cdot,\cdot)$, where
		$$
		h(x,t):= \begin{cases}\big(\mathcal{F}^{-1} \sqrt{\widehat k}\big)(x/t), & t\leq 1,\\
		0, & t>1.
		\end{cases}
		$$
		The {covariance structure of the white noise is given by} $\E W(K_1)W(K_2)=\int_{K_1\cap K_2}y^{-{(d+1)}}(1-y^\delta)dxdy$ for, say, compact subsets $K_1,K_2\subset \R^d\times (0,1)$.
		One may note that $h$ is supported in a cone $\{ (x,t)\in \R^d\times (0,\infty)\;: \;  |x|\leq ct,\; t\leq 1\}$. See in this connection also Lemma~\ref{le:cascadelike} below.  The \lq almost\rq\    in the notion of almost $\star$-scale invariant refers to this mild extra weight in the hyperbolic white noise and is visible in the covariance structure \eqref{eq:covLR} through the $1-e^{-\delta u}$-term.  \enddef
	\end{remark}
	\begin{remark}\label{rem:1} The reason why we restrict ourselves here to rotationally invariant functions is  that for the $\star$-scale invariant log-correlated field {$Y$} the function $g$ in the representation \eqref{eq:cov} is not in general  continuous at the diagonal, but it is so in case $k$ is rotationally invariant. 
		
		As suggested to us by an anonymous reviewer, we now provide an example of such a phenomenon. Let us consider the case $d=2$, and let $k:\R^2\to \R$ be given by $k(x,y)=k_0(x)k_0(2y)$, where we assume that $k_0:\R\to \R$ is H\"older continuous, $k_0(0)=1$, $k_0$ decays algebraically at infinity, and gives rise to a translation invariant covariance $k_0(x-x')$ on $\R^2$.	To prove that the function $g$ from \eqref{eq:cov} is not continuous up to the diagonal, it is sufficient to prove that the limit
		\[
		\lim_{r\to 0^+}g(r(e_1,e_2),0)=\lim_{r\to 0^+}\left[\int_0^\infty k(e^u r(e_1,e_2))du-\log r^{-1}\right]
		\]
		depends on the unit vector $(e_1,e_2)\in \R^2$. Comparing the cases $(e_1,e_2)=(1,0)$ and $(e_1,e_2)=(0,1)$, we find with our definitions
		\begin{align*}
		\int_0^{\infty}k(e^{u}r(1,0))du-\int_0^{\infty}k(e^{u}r(0,1))du&=\int_0^\infty \left(k_0(e^{u}r)- k_0(2e^{u}r)\right)du\\
		&=\int_0^{\log 2}k_0(e^u r)du\\
		&\to \log 2
		\end{align*}
		as $r\to 0$, and hence $g$ is not continuous up to the diagonal.
		\enddef
	\end{remark}
	
	Having in mind applications to analytic continuation of multiplicative chaos and to Onsager type inequalities,
	we record explicitly the following almost `cascade-like' independence  property.
	\begin{lemma}\label{le:cascadelike}
		If in Proposition \ref{pr:split} the covariance function $k$ is supported on $B(0,1)$, then the fields {$Y_t$} have the following independence of increments property: for $|x-y|\geq e^{-t}$, one has {$Y_u(x)-Y_t(x)\perp Y_s(y)$} for all $u\geq t$ and $s>0$. Especially,
		$$
		\E {Y_u(x)Y_{u'}(y)} = \E {Y_t(x)Y_{t}(y)}\quad\textrm{if}\quad u,u'\geq t \quad\textrm{and}\;\; |x-y|\geq e^{-t}.
		$$
		Similar statements hold  true for the fields {$L_t$}.
	\end{lemma}
	\begin{proof}
		This is an immediate consequence of the covariance structure \eqref{eq:covLR} 
	\end{proof}
	\subsection{Decomposing log-correlated fields}
	
	The following theorem contains Theorem {B} from the introduction, and it can be used both in order to construct locally log-correlated fields with a given covariance where $g$ is almost arbitrary, or in order to split a given log-correlated field  into a sum of a regular field and an (almost) $\star$-scale invariant field.

	\begin{theorem}\label{th:deco} Let $C(x,y)=\log(1/|x-y|) +g(x,y)$, where $g \in H^{s}_{loc}(U\times U)$  for some $s>d$. Fix $\delta>0$ let $k:\R^d\to\R$ be any function satisfying the conditions of Proposition \ref{pr:split} {\rm(}also the assumption of $k\in H_{loc}^s(\R^d)$ from the statement {\rm(vi)}{\rm)} with the additional assumptions that $k\in L^1(\R^d)$ and $\int_{\R^d}k >0$. Then for all small enough $\varepsilon >0$  it holds that:
		
		\smallskip
		
		\noindent {\rm (i)}\quad The function $C_{|B(0,\varepsilon)\times B(0,\varepsilon)}$ is the covariance of a log-correlated field $X$ on the ball $B(0,\varepsilon)$. 
		
		\smallskip
		
		\noindent {\rm (ii)}\quad In the neighbourhood $B(0,\varepsilon)$ {\rm (}a copy of{\rm )} the field $X$ can be written as the sum of independent Gaussian fields
		$$
		X={ L + R},
		$$
		where {$R$} has almost surely H\"older-continuous realizations, and {$L$} is the almost $\star$-scale invariant field from Proposition \ref{pr:split} obtained by using a dilation of the seed covariance $k(\lambda_0\cdot)$ with a suitable dilation factor $\lambda_0 \geq 1.$
	\end{theorem}
	\begin{proof} 
		We shall base the proof on a couple of auxiliary lemmas that will be used to show that one may add any (small enough and smooth enough) function to the covariance of the field $S$ in Proposition~\ref{pr:split} and still obtain a covariance.
		
		\begin{lemma}\label{le:888} {\rm (i)}\quad Assume that $h\in H^{s}(\R^{2d})$, where $s>0$. Then  for all real test functions $\varphi\in C_0^{\infty}(\R^d)$ it holds that
			$$
			\Big|\int_{\R^d\times \R^d} h(x,y)\varphi(x)\varphi(y)dxdy\Big| \leq \| h \|_{H^{s}(\R^{2d})}  \|\varphi\|_{H^{-s/2}(\R^{d})}^2.
			$$
			\noindent {\rm (ii)}\quad Let $\delta >0$ and the field $S$ be as in Proposition \ref{pr:split}. The covariance kernel $C_S(x,y)$  satisfies 
			$$
			\int_{\R^{2d}}C_S (x,y)\varphi(x)\varphi(y)\geq c\|\varphi \|_{H^{-d/2-\delta/2}(\R^d)}^2,
			$$
			where $c>0$ is a positive constant.
		\end{lemma}
		\begin{proof}
			(i)\quad The quantity on the left hand side  equals
			\begin{align*}
			&\Big|\int_{\R^d\times \R^d} \widehat h(\xi_1,\xi_2)\overline{\widehat\varphi(\xi_1)}\overline{\widehat\varphi(\xi_2)}d\xi_1d\xi_2\Big|\\
			&= \Big|\int_{\R^d\times \R^d} \widehat h(\xi_1,\xi_2)(|\xi_1|^2+1)^{s/4}(|\xi_2|^2+1)^{s/4}\frac{\overline{\widehat\varphi(\xi_1)}}{(1+|\xi_1|^2)^{s/4}}\frac{\overline{\widehat\varphi(\xi_2)}}{(1+|\xi_2|^2)^{s/4}}d\xi_1d\xi_2\Big|,
			\end{align*}
			and the result follows by Cauchy--Schwarz and by noting that $(|\xi_j|^2+1)\leq (1+|\xi_1|^2+|\xi_2|^2)$.

			(ii) The assumption that $k$ is integrable with $\int_{\R^d}k>0$  yields that $\widehat k$ is continuous with $\widehat k(0)>0$. We may thus pick $b>0$ so that $\widehat k(\xi)\geq b$ if $|\xi|\leq b$. By the covariance property of $k(x-y)$ we also have $\widehat k \geq 0$ everywhere. As we have $C_S(x,y)= H(x-y)$ with $H(x):=\int_0^\infty k(e^ux)e^{-\delta u}du$, it follows that 
			\begin{align*}
			\widehat H(\xi)= \int_0^\infty e^{-(d+\delta)u}\widehat k(e^{-u}\xi)du \geq b\int_{\log^+ (|\xi|/b)}^\infty e^{-(d+\delta)u}du \geq c(1+|\xi|)^{-d-\delta}
			\end{align*}
			with $c>0.$ We may then estimate 
			\begin{align*}
			\int_{\R^{2d}}C_S (x,y)\varphi(x)\varphi(y) \; &= \; \int_{\R^{d}}\overline{\varphi(x)}\left(\int_{\R^d}C_S (x,y)\varphi(y)dy\right)dx \\
			\;&=\; 
			\int_{\R^{d}}\widehat H(\xi)|\widehat \varphi(\xi)|^2d\xi \\
			&\gtrsim\int_{\R^{d}}(1+|\xi|)^{-d-\delta}|\widehat \varphi(\xi)|^2d\xi  \; \gtrsim\; \ \|\varphi \|_{H^{-d/2-\delta/2}(\R^d)}^2.
			\end{align*}
		\end{proof}

		\begin{lemma}\label{le:444}  Let $V\subset \R^{2d}$ be a neighbourhood of the origin and assume that $\delta\in (0,1/2)$. Assume that   $F\in H^{d+2\delta}(V)$ satisfies $F(0)=0$, and let $\psi\in C^\infty_c(\R^{2d})$ be supported in $B(0,1)$. 
			Then 
			\[ \lim_{\varepsilon\to 0^+}\| \psi(\cdot / \varepsilon)F\|_{H^{d+\delta}(\R^{2d})} = 0.\]
		\end{lemma}
		\begin{proof} 
			For a given $\varepsilon >0$ let  us denote by $T_\varepsilon$ the linear operator
			$$
			(T_\varepsilon f)(x) =\psi(\frac{x}{\varepsilon})\big(f(x)-f(0)\big).
			$$
			For fixed $\varepsilon >0$ it   is obviously bounded on any Sobolev space $H^s(\R^{2d})$ with $s>d$ since point evaluations are then continuous on  $H^s(\R^{2d})$ by the Sobolev embedding \eqref{eq:trois}. Towards our claim we first prove that there is a finite constant $C$, independent of $\varepsilon >0$, so that
			\begin{equation}\label{eq:une}
			\| T_\varepsilon f\|_{H^d(\R^{2d})}\leq C\|  f\|_{H^{d+\delta}(\R^{2d})}\qquad\textrm{for all}\quad 
			f\in H^{d+\delta}(\R^{2d}),
			\end{equation}
			and 
			\begin{equation}\label{eq:deux}
			\| T_\varepsilon f\|_{H^{d+1}(\R^{2d})}\leq C\|  f\|_{H^{d+1+\delta}(\R^{2d})} \qquad\textrm{for all}\quad 
			f\in H^{d+1+\delta}(\R^{2d}).
			\end{equation}

			We first consider  \eqref{eq:une}, for which it is enough to bound $\| T_\varepsilon f\|_{L^2(\R^{2d})}$ and $\| D^dT_\varepsilon f\|_{L^2(\R^{2d})}$. Again by  \eqref{eq:trois} we have the obvious bound $\| T_\varepsilon f\|_{L^2(\R^{2d})}\lesssim \|  f\|_{H^{d+\delta}(\R^{2d})}$, and the $d$:th derivative can be estimated by
			\begin{align}\label{eq:quatre}
			\int_{\R^{2d}}|D^d(T_\varepsilon f)|^2\lesssim&\;\; \varepsilon^{-2d}\int_{|x|\leq \varepsilon}|f(x)-f(0)|^2dx\; +\; 
			\sum_{k=1}^d \varepsilon^{2k-2d}\int_{|x|\leq \varepsilon}|D^kf(x)|^2dx\; \nonumber\\
			\lesssim&\;\; \|  f\|_{H^{d+\delta}(\R^{2d})}^2\;+ \;\sum_{k=1}^d \varepsilon^{2k-2d}\Big(\int_{|x|\leq \varepsilon}|D^kf(x)|^{q_k}dx\Big)^{2/q_k}
			\Big(\int_{|x|\leq \varepsilon}1dx\Big)^{1-2/q_k}\nonumber\\
			\lesssim&\;\;  \|  f\|_{H^{d+\delta}(\R^{2d})}^2+ \|  f\|_{H^{d}(\R^{2d})}^2\;\lesssim\; \|  f\|_{H^{d+\delta}(\R^{2d})}^2. 
			\end{align}
			Here we applied the Sobolev embedding  \eqref{eq:six} on  (components of) the function $D^kf$ with the exponent $q=q_k$ such that $\displaystyle\frac{1}{q_k}=\frac{1}{2}-\frac{(d-k)}{2d}$ so that $ 1-2/q_k=(d-k)/d.$ 
			This yields \eqref{eq:une}. In turn, we proceed similarly to obtain
			\begin{align}\label{eq:cinq}
			\notag \int_{\R^{2d}}|D^{d+1}(T_\varepsilon f)|^2&\lesssim \varepsilon^{-2d-2}\int_{|x|\leq \varepsilon}|f(x)-f(0)|^2dx\\
			&\qquad \; +\; 
			\sum_{k=1}^{d+1} \varepsilon^{2k-2d-2}\int_{|x|\leq \varepsilon}|D^kf(x)|^2dx\nonumber \\
			&\lesssim\|  f\|_{H^{d+1+\delta}(\R^{2d})}^2\\
			&\qquad \; { +} \; \sum_{k=1}^{d+1} \varepsilon^{2k-2d-2}\Big(\int_{|x|\leq \varepsilon}|D^kf(x)|^{\widetilde q_k}dx\Big)^{2/\widetilde q_k}
			\Big(\int_{|x|\leq \varepsilon}1dx\Big)^{1-2/\widetilde q_k}\nonumber\\
			\notag &\lesssim\;\; \|  f\|_{H^{d+1+\delta}(\R^{2d})}^2,
			\end{align}
			where the first part was handled by the fact that the latter embedding in \eqref{eq:trois} yields a uniform Lipschitz bound on $f$. This time we applied \eqref{eq:six} with the exponents
			$\displaystyle\frac{1}{\widetilde q_k}=\frac{1}{2}-\frac{(d+1-k)}{2d}$ so that $ 1-2/\widetilde q_k=(d+1-k)/d$ for $1\leq k\leq d+1.$
			
			By applying interpolation \eqref{eq:interpolation} on the inequalities \eqref{eq:une} and \eqref{eq:deux}  with respect to the smoothness index we obtain that $\| T_\varepsilon f\|_{H^{d+\delta}(\R^{2d})}\leq C\|  f\|_{H^{d+2\delta}(\R^{2d})}$ with a uniform bound with respect to $\varepsilon$. In order to apply this  to prove that the stated limit is zero, it is now enough to use the density of smooth functions  to be able to assume that $F\in C^\infty_c(\R^{2d})$ with $F(0)=0$. In that case we see immediately from \eqref{eq:quatre} that $\| T_\varepsilon F\|_{H^d(\R^{2d})}\to 0$ as $\varepsilon\to 0^+$. As we already know that at the same time $\| T_\varepsilon F\|_{H^{d+1}(\R^{2d})}$ stays bounded, again interpolation (now for a fixed single function, see \eqref{eq:interpolation2})  implies that   $\| T_\varepsilon F\|_{H^{d+\delta}(\R^{2d})}\to 0$ as $\varepsilon\to 0^+$.
		\end{proof}
		
		We now continue the proof of the Theorem~\ref{th:deco} and consider first part (ii). We write $C_X=\log(1/|x-y|)+g(x,y)$. By Proposition \ref{pr:split}, for our given function $k$, there exists a $\star$-scale invariant field {$Y$} in some neighbourhood of the origin for which we have ${C_Y(x,y)}=\log(1/|x-y|)+g_0(x,y)$, where $g_0\in H^s_{loc}(\R^{2d})$ with $s>d$. Let {$Y,S,L$} stand for the fields from Proposition \ref{pr:split}, where the construction is performed by using the dilated seed covariance function $k(\lambda_0\cdot )$, where the dilation factor $\lambda_0$ will be determined in the course of the proof. We may write
		\begin{align}\label{eq:write1}
		C_X(x,y)= {C_Y(x,y)}+ (g-g_0)(x,y)={C_{L}(x,y)} + \big(C_S(x,y)+(g-g_0)(x,y)\big).
		\end{align}
		Assume first that $a:=g(0,0)-g_0(0,0)\geq 0$. Then we set $F(x,y)= (g-g_0)(x,y)-a $ and pick a non-negative test function $\psi\in C_c^\infty(\R^{2d})$   such that $\psi_{|B(0,2)}(x)=1$  and $\psi=0$ outside $B(0,3)$. Since in a suitable neighbourhood $V$  of the origin $F\in H^s_{loc}(V)$ with $s>d$, with $F(0,0)=0$, by combining  Lemmas 
		\ref{le:888} and \ref{le:444} we see that $C_S +\psi_\varepsilon F$ is positive definite for small enough $\varepsilon >0$. Especially, then $C_S+F$ is a H\"older continuous covariance kernel  on $B(0,\varepsilon)$  by \eqref{eq:trois}, and we may construct on $B(0,\varepsilon)$ a centred Gaussian field {$R_0$} with ${C_{R_0}}=C_S+F$  with a.s. H\"older continuous realizations. Let $G\sim N(0,1)$ be a Gaussian variable independent from {$R_0$} and {$L$}. Set ${R:=R_0}+\sqrt{a}G$. It follows that on  $B(0,\varepsilon)\times B(0,\varepsilon)$ it holds that
		$$
		C_X={C_{L}+C_{R_0}+a =C_{L}+ C_{R}},
		$$
		which yields the claim in the case $a\geq 0$ (thus in this case we may take $\lambda_0=1$).
		
		The  case $a:=g(0,0)-g_0(0,0)<0$ can be reduced to the previous case if we show that by replacing $k$ by 
		the dilation $k(\lambda_0\cdot )$ we may decrease $g_0(0,0)$ as much as we want. Namely, denote by $\widetilde g_0$ the function $g_0$ after the replacement, and note that directly from formula \eqref{eq:logno} we obtain that
		\begin{align*}
		g_0(0,0)-\widetilde g_0(0,0) &= \lim_{x\to 0} \left( \int_0^\infty k(e^{u}x)du-  \int_0^\infty k(e^{u}\lambda_0x)du\right)\\ &= \lim_{x\to 0}\int_0^{\log \lambda_0}k(e^{u}x)du =\log\lambda_0,
		\end{align*}
		from which the claim follows since we may choose $\lambda_0\geq e^{-a}.$ We have thus completed the proof of (ii),  which  clearly implies  (i).
	\end{proof}
	
	\begin{remark}\label{rem:sileys} We observe that the condition on smoothness of $k$  and $g$  are satisfied  if these functions have $C^{d+\varepsilon}_{loc}$-smoothness. However, the actual requirement of  $H^{d+\varepsilon}_{loc}$-smoothness is in a certain sense much less restrictive as it e.g. allows for local behaviour of type $|x-x_0|^\varepsilon$.   \enddef
		
	\end{remark}

	\section{Critical multiplicative chaos}\label{se:critical}

	We recall that the definition of critical chaos was given in Section \ref{sec:intro}. Critical chaos  appears in many facets of the multiplicative chaos theory, even when dealing with non-critical chaos.  In particular, it encodes the location of maxima of the log-correlated Gaussian field and it appears as a building block for the super-critical chaos \cite{RhVa,MRV,BRV}.

	As noted in the introduction, existence of the correctly normalized critical chaos has been proven only in the setting of $\star$-scale invariant fields (again, we refer to \cite{DRSV1,DRSV2,JS,P} here), and our goal is to extend this to a more general class of log-correlated fields. The basic idea will be to use Theorem A to reduce to the $\star$-scale invariant case and apply the known results of \cite{DRSV1,DRSV2,JS,P}. Before going into this, we will introduce the notion of convergence of random measures which is relevant to the critical case.

	\medskip
	
	\begin{definition}\label{def:weak*probability}
		Let $K\subset\R^d$ be a compact subset, and let $\mu,\mu_1,\mu_2,\ldots$ be random Borel probability measures  on $K$ (defined on a common probability space). We say that $\mu_n\to \mu$ weak$\hbox{}^*$ in probability, denoted by 
		$\displaystyle\mu_n\weaklyprob\mu$, if for every $\phi\in C(K)$
		one has 
		\begin{equation}\label{eq:integrals}
		\displaystyle\int_K\phi(x)\mu_n(dx)\inprob\displaystyle\int_K\phi(x)\mu(dx). 
		\end{equation} \enddef
		
	\end{definition}
	
	We collect in the following lemma some basic properties of weak$\hbox{}^*$ convergence in probability. Below, when we speak of convergence of continuous random functions in probability, we refer to convergence in probability with respect to the sup-norm unless others stated. We note that as a compact set $K\subset\R^d$ is separable, also $C(K)$ is separable.
	\begin{lemma}\label{le:conv1} Assume that $\mu,\mu_1,\mu_2,\ldots$ are random Borel probability measures  on a compact subset $K\subset\R^d.$  
		
		\smallskip
		
		\noindent  {\rm (i)}\quad Assume that $\mu_n\weaklyprob\mu$ over a compact set $K$. Then the the total masses $\mu_n(K)$ stay uniformly bounded in probability, i.e.
		$$
		\Prob \{ \mu_n(K)>\lambda \;\;\textrm{\rm for some}\; n\} \to 0 \quad\textrm{as}\quad \lambda\to\infty.
		$$
		
		\noindent {\rm (ii)}\quad If \eqref{eq:integrals} holds true for elements in a countable dense set $\{\phi_j\}_{j=1}^\infty\subset C(K)$, then it follows that $\mu_n\weaklyprob\mu$.
		
		\smallskip
		
		\noindent {\rm (iii)}\quad Assume that $f, f_n\in C(K)$  $(n\geq 1)$ are random continuous functions on $K$ with the property that $f_n\to f$ in probability in $C(K)$, and assume that  $\mu_n\weaklyprob\mu$.  Then
		$$
		f_n\mu_n\weaklyprob f\mu\quad\textrm{as}\quad n\to\infty.
		$$
	\end{lemma}
	\begin{proof} (i)\quad This follows easily by  applying   \eqref{eq:integrals} with the choice $\phi\equiv 1$. In turn (ii) follows by approximating given $\phi\in C(K)$ up to $\varepsilon$ by suitable $\phi_j$ and invoking part (i) of the lemma. Finally, (iii) follows  by a similar argument.
	\end{proof}
	
	The following theorem generalizes the result on the existence of critical chaos \cite{DRSV2} for a large class of log-correlated fields, and in particular verifies that $\star$-scale invariance is not needed a priori.  The mollification $X_\varepsilon:=\psi_\varepsilon*X$ in the stated result below may be performed by using any compactly supported and non-negative  smooth test function with integral 1. 
	
	\begin{theorem}\label{th:critical}
		Let $X$ be a log-correlated Gaussian field with a covariance given by the decomposition \eqref{eq:cov} with $g\in H^{s}_{loc}(U\times U)$. Then the corresponding critical chaos exists, i.e. there is a locally finite random Borel measure  $ \mu_{\sqrt{2d}}$ on $U$  such that for a  mollification $X_{\varepsilon}$ of the field $X$ we have the convergence 
		\begin{equation}\label{eq:crit_convergence}
		(\log(1/\varepsilon))^{1/2}\exp \Big(\sqrt{2d}X_\varepsilon (x)-d\, \E X_\varepsilon (x)^2\Big)dx\weaklyprob  \mu_{\sqrt{2d}}
		\end{equation}
		over any compact subset $K \subset U$. The random limit measure $\mu_{\sqrt{2d}}$ is independent of the mollification used.
	\end{theorem}
	\begin{proof}
		{ According to Theorem A we may write 
			$$
			X=Y+R,
			$$
			where $Y$ is a $\star$-scale invariant field for which \cite[Theorem 5]{DRSV2} holds and $R$ is a Gaussian field with a.s. locally H\"older continuous realizations.} Hence the critical chaos constructed from the approximations {$Y_t$} of {}$Y$ converges to a limit measure on $U$ as $t\to\infty$, let us call it $\nu_{\sqrt{2d}}$. Proposition \ref{pr:split} verifies that for any sequence $t_n\nearrow \infty$ the approximations {$Y_{t_n}$} of the field $Y$ satisfy the conditions of \cite[Theorem 1.1. and {Theorem} 4.4]{JS}, whence we deduce  the convergence in probability for the standard convolution approximations (we write {$(Y)_\varepsilon$} to denote a convolution approximation of {$Y$} -- this is to avoid confusion with {$Y_t$} which referred to the approximation of {$Y$} from Proposition \ref{pr:split})
		\begin{equation}\label{eq:convY}
		(\log(1/\varepsilon))^{1/2}\exp \Big(\sqrt{2d}({Y})_\varepsilon (x)-d\,\E ({Y})_\varepsilon (x)^2\Big)dx\weaklyprob \nu_{\sqrt{2d}} \qquad \textrm{on}\quad { U}.
		\end{equation}
		We may then write
		\begin{align*}
		&(\log(1/\varepsilon))^{1/2}\exp \Big(\sqrt{2d}X_\varepsilon (x)-d\,\E X_\varepsilon (x)^2\Big)dx
		\\
		&= f_\varepsilon(x)(\log(1/\varepsilon))^{1/2}\exp \Big(\sqrt{2d}({Y})_\varepsilon (x)-d\,\E ({Y})_\varepsilon (x)^2\Big)dx,
		\end{align*}
		where the random continuous function $f_\varepsilon$ on { $U$} is given by the expression
		$$
		f_\varepsilon(x):=\exp\big(d(\E ({Y})_\varepsilon (x)^2-\E X_\varepsilon (x)^2)\big)\exp\big(\sqrt{2d} (R)_\varepsilon(x)\big).
		$$
		{ Let $K \subset U$ be compact.} Obviously { $(R)_\varepsilon \to R$ in probability in $C(K)$}, so in view of Lemma \ref{le:conv1} (iii), to prove the stated convergence on { $K$} it only remains to prove that $\E ({Y})_\varepsilon (x)^2-\E X_\varepsilon (x)^2$ converges uniformly on { $K$} as $\varepsilon\to 0.$ 
		{ This however follows simply because both ${Y}$ and $X$ have covariances of the form $\log \frac{1}{|\cdot - \cdot|}$ \textit{plus} some H\"older continuous functions $g^Y$ and $g^X$, respectively. Therefore
			\[\E (Y)_\varepsilon (x)^2 - \E X_\varepsilon(x)^2 = ((\psi_\varepsilon \otimes \psi_\varepsilon)*(g^Y - g^X))(x,x),\]
			which clearly converges uniformly to $g^Y - g^X$. This finishes the proof.}
	\end{proof}
	
	We will next show that one can also construct the critical chaos via the so called derivative normalization, which is obtained by taking the derivative $-\frac{d}{d\beta}|_{\beta = \sqrt{2d}} e^{\beta X_\varepsilon(x) - \frac{\beta^2}{2} \E X_\varepsilon(x)^2}$ {and letting $\varepsilon\to 0$}.
	
	\begin{theorem}\label{th:derivative_martingale}
		Let $X$ be a log-correlated Gaussian field with a covariance given by the decomposition \eqref{eq:cov} with $g \in H^s_{loc}(U \times U)$ for some $s > d$. Then {for any compact $K\subset U$,} the derivative renormalization measures
		\[(-X_\varepsilon(x) + \sqrt{2d} \E X_\varepsilon(x)^2) e^{\sqrt{2d} X_\varepsilon(x) - d \E X_\varepsilon(x)^2}\]
		for the mollifications $X_\varepsilon$ of the field $X$ converge weak$^*$ in probability {over $K$} as $\varepsilon \to 0$ to $\sqrt{\pi/2} \mu_{\sqrt{2d}}$, where $\mu_{\sqrt{2d}}$ is the critical chaos measure given in Theorem~\ref{th:critical}.
	\end{theorem}
	
	\begin{proof}
		{ We again use Theorem~A and write $X = Y + R$ with $Y$ a $\star$-scale invariant field which this time satisfies the assumptions of \cite[Theorem~1.2.]{P}.}
		Thus, \cite[Theorem~1.2.]{P} yields that  the derivative renormalization measures constructed from $({Y})_\varepsilon$ converge weak$^*$ in probability to $\sqrt{\pi/2}\nu_{\sqrt{2d}}$, where $\nu$ is as in the proof of Theorem~\ref{th:critical}. Next we split the renormalization in two parts by writing
		\begin{align*}
		& (-X_\varepsilon(x) + \sqrt{2d} \E X_\varepsilon(x)^2) e^{\sqrt{2d} X_\varepsilon(x) - d \E X_\varepsilon(x)^2} \\
		= & (-({Y})_\varepsilon(x) + \sqrt{2d} \E ({Y})_\varepsilon(x)^2) e^{\sqrt{2d} ({Y})_\varepsilon(x) - d \E ({Y})_\varepsilon(x)^2} e^{\sqrt{2d} (R)_\varepsilon(x) - d ( \E X_\varepsilon(x)^2 - \E ({Y})_\varepsilon(x)^2 )} \\
		& + (-(R)_\varepsilon(x) + \sqrt{2d} (\E X_\varepsilon(x)^2 - \E ({Y})_\varepsilon(x)^2)) e^{\sqrt{2d} X_\varepsilon(x) - d \E X_\varepsilon(x)^2}.
		\end{align*}
		{ On any compact set $K \subset U$} the second term goes to $0$ weak$^*$ in probability (by e.g. Theorem~\ref{th:critical}) since the factor $(-(R)_\varepsilon(x) + \sqrt{2d} (\E X_\varepsilon(x)^2 - \E ({Y})_\varepsilon(x)^2))$ is almost surely uniformly bounded  and converges uniformly {in probability} as $\varepsilon\to 0$, a fact which can be deduced as in the proof of Theorem~\ref{th:critical}. Similarly, the first term converges to $\sqrt{\pi/2} \mu_{\sqrt{2d}}$ as we wanted.
	\end{proof}
	
	{
		Finally we note\footnote{We thank A. Sep\'ulveda for asking us a question that led to  this application.}  that in two dimensions the critical chaos can also be seen as a suitable normalized limit of subcritical chaoses as $\beta \to 2$ along the real axis. As opposed to the two previous results, we don't choose now as a reference field a $\star$-scale invariant one, but the Gaussian free field (GFF) with Dirichlet boundary conditions in a suitable planar domain -- more precisely, it is the centered log-correlated Gaussian field whose covariance is given by the Dirichlet Green's function of the domain. As first conjectured in \cite[Conjecture 9]{DRSV1}, it is natural to expect a version of the following result to be true in all dimensions, but we lack a suitable reference field for which the result would have been proven.

		\begin{theorem}\label{thm:subcritical_to_critical}
			Let $X$ be a log-correlated Gaussian field on a planar domain $U$ with a covariance given by the decomposition \eqref{eq:cov} with $g \in H^s_{loc}(U \times U)$ for some $s > 2$. Then {over any compact $K\subset U$} we have the convergence in probability
			\[\lim_{\beta \nearrow 2} \frac{\mu_\beta}{2 - \beta} = \sqrt{2\pi} \mu_{2}\]
			in the space of non-negative Radon measures under the weak$^*$-topology,
			where $\mu_{2}$ is the critical chaos measure given in Theorem~\ref{th:critical} and 
			\[
			\mu_\beta(dx) = \lim_{\varepsilon \to 0} e^{\beta X_\varepsilon(x) - \frac{\beta^2}{2} \E X_\varepsilon(x)^2} \, dx\]
			is the subcritical chaos measure with parameter $\beta < 2$.
		\end{theorem}
		
		\begin{proof}[Proof of Theorem~\ref{thm:subcritical_to_critical}]
			{If we have a finite number of random variables, each converging in probability, their sum also converges in probability, so using regularity of the critical measure} it is enough to consider convergence in an open ball $B := B(x_0,r) \subset U$ such that $B(x_0,2r) \subset U$. Then by Theorem~A we may inside $B$ write $X = L + R$, where $L$ is a GFF on $B(x_0,2r)$ with Dirichlet boundary conditions, and $R$ is a H\"older-regular field. Let
			\[d\mu_\beta^{GFF}(x) = \lim_{\varepsilon \to 0} e^{\beta L_\varepsilon(x) - \frac{\beta^2}{2} \E L_\varepsilon(x)^2} \, dx\]
			and let $\mu_2^{GFF}(x)$ be the corresponding critical chaos. Then by \cite[Theorem~1.1]{APS} we have
			\[\frac{\mu_\beta^{GFF}}{2 - \beta} \to \sqrt{2\pi} \mu_2^{GFF}\]
			as $\beta \nearrow 2$, with convergence weak$^*$ in probability {(over $\overline{B}$)}. Clearly also, as $\beta\nearrow 2$,
			\[e^{\beta R(x) - \frac{\beta^2}{2} \E R(x)^2 - \beta^2 \E R(x) L(x)} \to e^{2 R(x) - 2 \E R(x)^2 - 4 \E R(x)L(x)}\]
			in probability in $C({\overline{B}})$. Here $\E R(x) L(x)$ can be understood as the limit as $\varepsilon \to 0$ of
			\[\frac{1}{2} (\E X_\varepsilon(x)^2 - \E L_\varepsilon(x)^2 - \E R_\varepsilon(x)^2).\]
			The claim now follows from Lemma~\ref{le:conv1}.
		\end{proof}
	}
	\begin{remark}\label{rem:scmax}
		We note that based on our result in the critical case, a natural question a reader familiar with multiplicative chaos literature might have is can we prove something similar in the supercritical case and what can we say about the maximum of the field. While there are indeed things that can be proven about supercritical multiplicative chaos and the maximum of the field using our decomposition results, we fear that obtaining as precise results as in the critical case would require a significant amount of work. Nevertheless, we sketch here a few arguments concerning supercritical chaos utilizing our Theorem A.
		
		\medskip
		
		As proven in \cite[Corollary 2.3]{MRV}, given a $\star$-scale invariant log-correlated field $Y$ on some domain $U\subset \R^d$, with $Y_t$ being the $\star$-scale invariant cut-off parametrized such that $\E Y_t(x)^2=t$, and under suitable regularity assumptions on the associated seed covariance $k$, for $\beta>\sqrt{2d}$, the family of random measures
		\[
		t^{\frac{3\beta}{2\sqrt{2d}}}e^{t(\frac{\beta}{\sqrt{2}}-\sqrt{d})^2}e^{\beta Y_t(x)-\frac{\beta^2}{2}t}dx
		\]
		converges in law, as $t\to \infty$,  with respect to the weak convergence of measures to a non-trivial purely atomic limiting measure, whose law can be characterised explicitly in terms of the law of the critical measure -- a property known as freezing. We refer readers interested in further details to \cite{MRV}.
		
		Now given an arbitrary log-correlated field $X$ on $U$, we wish to construct a random measure which would give a precise definition of $e^{\beta X(x)}dx$ for $\beta>\sqrt{2d}$ in the same sense as for $Y$. To do this, we use Theorem A to write $X=Y+G$ for some H\"older-continuous Gaussian field $G$. We then introduce the following, rather non-canonical, approximation of $X$: for $t>0$, let $X_t(x):=Y_t(x)+G(x)$. With some effort, which we choose not to document here as we suspect it to be of little use to the reader, one can then convince oneself (using of course \cite[Corollary 2.3]{MRV}) that 
		\[
		t^{\frac{3\beta}{2\sqrt{2d}}}e^{t(\frac{\beta}{\sqrt{2}}-\sqrt{d})^2}e^{\beta X_t(x)-\frac{\beta^2}{2}t}dx
		\]
		converges in law (with respect to the weak topology) as $t\to \infty$ to something.
		
		Thus what Theorem A can be used for in this setting is constructing a candidate for what $e^{\beta  X(x)}dx$ should mean for $\beta>\sqrt{2d}$. The drawback being that the construction is rather non-canonical in that it involves coupling to an arbitrary $\star$-scale invariant field, and it does not seem to us that it is obvious that the law of the limiting measure is the same for all $\star$-scale invariant fields. It is also not clear how easily one can get a hold of basic properties of the measure such as freezing. From the point of view of applications, from say random matrix theory, where the associated field is Gaussian only asymptotically in some sense, it would be more satisfying if one had a construction for the supercritical measure say in terms of convolution approximations instead of a coupling to $\star$-scale invariant ones as these might be impossible to realize in the non-Gaussian setting. We suspect that resolving such uniqueness and freezing questions requires a non-trivial amount of further work.
		
		The difficulties one runs into when trying to use our decomposition results to relate the maximum of a general log-correlated field to the maximum of say a $\star$-scale invariant one are rather similar in nature. We omit further discussion on this and leave formulating precise statements to the reader.  \enddef
	\end{remark}
	
	\section{Analytic dependence on $\beta$}
	
	J. Barral \cite{B} made the important observation that evaluations of subcritical cascade measures against test functions continue analytically in the intermittency parameter $\beta$ to the the domain
	\begin{equation}\label{eq:holdomain}
	A\; :=\; \textrm {span}\big(\pm\sqrt{2d}\cup B(0,\sqrt{d})\big),
	\end{equation}
	i.e., to the open domain that is the {union} of the ball $B(0,\sqrt{d})$ and the quadrilateral domain defined by  the four lines passing through points $\pm\sqrt{2d}$ at angles $\pm \pi/4$. This is illustrated in Figure~\ref{fig:the_eye}.
	
	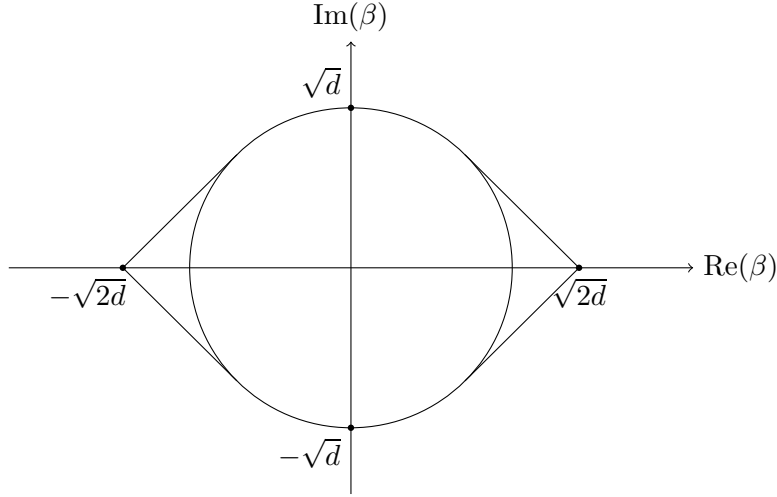
\begin{figure}[h]
		\centering
		\begin{tikzpicture}[x=1.5cm,y=1.5cm]
		\def\sqrttwo{1.4142135623730951}
		
		\draw[->] (-3,0) -- (3,0) node[right] {$\Re(\beta)$};
		\draw[->] (0,-2) -- (0,2) node[above] {$\Im(\beta)$};
		
		\draw (0,0) circle (\sqrttwo);
		\draw (1,1) -- (2,0) -- (1,-1);
		\draw (-1,1) -- (-2,0) -- (-1,-1);
		\fill (0,\sqrttwo) circle (1.2pt) node[anchor=south east] {$\sqrt{d}$};
		\fill (0,-\sqrttwo) circle (1.2pt) node[anchor=north east] {$-\sqrt{d}$};
		\fill (-2,0) circle (1.2pt) node[below left=0 and -5pt] {$-\sqrt{2d}$};
		\fill (2,0) circle (1.2pt) node[anchor=north] {$\sqrt{2d}$};
		\end{tikzpicture}
		
		\caption{The extended subcritical regime for complex $\beta$, namely the set $A$ from \eqref{eq:holdomain}. The circle corresponds to the $L^2$-phase.}\label{fig:the_eye}
	\end{figure}
	
	In the case of multiplicative chaos, this is easy to check in the $L^2$-range $|\beta|<\sqrt{d}$. In \cite[Appendix 1]{AJKS} it was noted, that the analytic dependence also holds for the Gaussian chaos constructed from certain $1$-dimensional essentially $\star$-scale invariant fields. The case of more general fields has remained an open question. 
	
	Our next result resolves positively this problem with the aid of the decomposition of Theorem \ref{th:deco}. For the reader's  convenience we shall give all the details of the argument here, although after Theorem \ref{th:deco} a considerable portion of the proof basically repeats the ideas in  \cite{B,AJKS} and \cite[Lemma 15]{J}.

	\begin{theorem}\label{th:holomorphic} Let $X$ be a log-correlated field in a domain $U\subset\R^d$  with the covariance structure \eqref{eq:cov}, where $g\in H^s_{loc}(U\times U)$   with $s>d$. Then for every $\beta\in A$ there exist  $H^{-d}_{loc}(U)$-valued  random variables $\mu_\beta$, all on the same probability space, which for  any  fixed $\beta\in (0,\sqrt{2d})$ almost surely agrees  with the standard definition of the chaos $``\exp(\beta X)"$. Moreover, for any $\psi\in C_c^\infty (U)$, almost surely  the map $\beta\mapsto\psi\mu_\beta$ is holomorphic in $A$ with values in $H^{-d}(\R^d)$.
	\end{theorem}
	\begin{proof} 
		We pick an arbitrary $x_0\in U$ and claim first that there is a neighbourhood $U_{x_0}:=B(x_0,r_{x_0})$ with $r_{x_0}\in (0,1/3)$ such that  for any fixed $\psi_{x_0}\in C^\infty_c(U_{x_0})$ there  exists a  random holomorphic $H^{-d}(\R^d)$-valued map 
		$$
		A\ni \beta\; \mapsto\; \eta_{x_0}(\beta)\in H^{-d}(\R^d),
		$$
		where for any fixed real value  of the inverse temperature $\beta\in \R\cap A$ there is  an almost sure agreement with the standard chaos measure:
		\begin{equation}\label{eq:coinc}
		\eta_{x_0}(\beta)="\psi_{x_0}\exp(\beta X)"\qquad \textrm{almost surely as elements in}\quad H^{-d}(\R^d). 
		\end{equation}
		
		Let $X= {L+R}$ (a.s.) be a decomposition\footnote{We may move originally to a copy of $X$ in a probability space where there is room for producing countably many such a.s. decompositions of $X$ simultaneously, as this will be required  for in the proof.} of $X$ on $U_{x_0}$ as in Theorem \ref{th:deco} (ii), where we choose $k$ to be smooth and supported on $B(0,1)$. Let $\beta_0\in A$ be arbitrary. We will first indicate how the claim can be deduced from the fact that for arbitrary $\beta_0\in A$ there exists an exponent $p=p(\beta_0)\in(1,2)$  and a radius $\delta:=\delta_{\beta_0}>0$ so that by denoting $B_{\beta_0}:=B(\beta_0,\delta)$ we have $\overline{B_{\beta_0}}\subset A$ and the uniform estimate
		\begin{equation}\label{eq:yks}
		\E \Big| \int_{U_{x_0}} \psi(x)\nu_{n,\beta}(x)dx \Big|^p \leq  C\|\psi\|_{L^\infty(U_{x_0})}^p, 
		\end{equation}
		holds true for all integers $n\geq 1$, $\beta\in \overline{B_{\beta_0}} $ and $\psi\in L^\infty(U_{x_0}),$ {with the constant $C$ independent of these quantities.} {Above we wrote}
		$$
		\nu_{n,\beta}(x):=  \exp\big( \beta {L_n}(x)-\frac{\beta^2}{2}\E {L_n}(x)^2\big), \qquad n\in \N,
		$$
		where ${L_n}$ was defined Proposition \ref{pr:split} (now $t=n$).
		We will postpone the proof of \eqref{eq:yks} and first show how it implies our claim.
		
		The inequality \eqref{eq:yks} transfers to
		\begin{equation}\label{eq:kaks}
		\E \Big| \int_{U_{x_0}} \psi(x)\mu_{n,\beta}(x)dx \Big|^p \leq  C\|\psi\|_{L^\infty(U_{x_0})}^p,
		\end{equation}
		with 
		\begin{align*}
		\mu_{n,\beta}(x)&:= \ \exp\Big( \beta \big({L_n}(x)+{R}(x)-\frac{\beta^2}{2}\E\big( { L_n}(x)+ {R}(x)\big)^2\Big)\\
		&= 
		\exp\big( \beta {R}(x)-\frac{\beta^2}{2}\E {R}(x)^2\big)\nu_{n,\beta}(x),
		\end{align*}
		which yields on $U_{x_0}$  an approximation   of the actual chaos  we are after. Namely, \eqref{eq:kaks} follows as one applies \eqref{eq:yks} to  $\exp\big( \beta {R}(x)-\frac{\beta^2}{2}\E {R}(x)^2\big)\psi(x)$ instead of $\psi$, conditions on ${L}$, and the fact that by Fernique's theorem  $\E \| \exp\big( \beta  {R}(x)-\frac{\beta^2}{2}\E {R}(x)^2\big)\|^p_{L^\infty(U_{x_0})}<\infty $ for all $p<\infty$. Clearly  $B_{\beta_0}\ni \beta\mapsto \psi_{x_0}\mu_{n, \beta}$ takes values in the ${H}^{-d}(\R^d)$ valued Hardy-space $\mathcal{H}^{p}(B_{\beta_0},H^{-d}(\R^d))$ (recall here the definition and basic properties of Hilbert space valued Hardy spaces from Section \ref{sec:spaces}). Thus, by construction, $(\psi_{x_0}\mu_{n, \beta})_{n\geq 1}$ is a $\mathcal{H}^{p}(B_{\beta_0}, H^{-d}(\R^d))$-valued  martingale, and its $p$-boundedness obviously  follows as soon as we establish the uniform pointwise estimate
		\begin{equation}\label{eq:kol}
		\E \|\psi_{x_0}\mu_{n, \beta}\|^p_{ H^{-d}(\R^d)}\leq C'\qquad\textrm{for all}\quad n\geq 1\quad \textrm{and}\quad \beta\in\partial B_{\beta_0}.
		\end{equation}
		\newcommand{\half}{\frac{1}{2}}
		In order to obtain this bound, by translation invariance we may assume that $x_0=0$. Since  then supp$(\psi_0)\subset [-1/3,1/3]^d,$  we may compute the Sobolev norm in terms of the  the standard Fourier coefficients\footnote{This is a well-known  fact, which follows easily from the definition for integer values of the smoothness index $s$, and generalizes  by interpolation to all $s\geq 0.$}.  In particular,  the concavity of $x\mapsto x^{p/2}$ for $x\geq 0$ (recall that $1<p<2$) and   \eqref{eq:kaks} yield
		\begin{align*} 
		&\E \|\psi_{x_0}\mu_{n, \beta}\|^p_{ H^{-d}(\R^d)}\lesssim \E\big( \sum_{k\in\Z^d}(1+|k|^2)^{-d}|\widehat{(\psi_{x_0}  \mu_{n, \beta})}(k)|^2\big)^{p/2}\\
		&\leq \sum_{k\in\Z^d}(1+|k|^2)^{-dp/2}\E|\widehat{(\psi_{x_0} \mu_{n, \beta})}(k)|^{p} < C.
		\end{align*}
		Hence we have a bounded martingale with values in the Hardy space, so again by \cite[Theorem 3.61, Theorem 1.95]{HNVW} we have almost sure convergence   in the Hardy space $\mathcal{H}^{p}(B(\beta_0,\delta),H^{-d}(\R^d))$. Finally, one observes by \eqref{eq:hardy bound}  that due to the convergence in the Hardy space we have almost sure uniform pointwise convergence in $B(\beta_0,\half \delta)$  of the sequence $\psi_{x_0}\mu_{\beta,n}$.  
		
		We may then cover $A$ by countably many such discs $B(\beta_0,\half \delta)$. It follows that almost surely the sequence $\psi_{x_0} \mu_{n, \beta}$ of analytic $H^{-d}(\R^d)$-valued functions on $A$ converges locally uniformly. We denote the almost sure limit, that is then an analytic $H^{-d}(\R^d)$-valued random function on the domain $A$ by
		$
		\eta_{x_0}.
		$
		The construction is completed as soon as we check \eqref{eq:coinc}. Thus, let $\beta\in (-\sqrt{2d},\sqrt{2d})=A\cap\R$. By Proposition \ref{pr:split}, $({L_n})_{n\geq1}$ yields a standard approximation of ${L}$, which implies that $({L_n + R})_{n\geq 1}$  provides one for $X$. It then follows from standard multiplicative chaos theory that $\eta_{x_0}(\beta)$ coincides with $"\psi_{x_0}\exp(\beta X)"$.

		In order to complete the proof of the claim we made in the beginning of the proof, it remains to prove \eqref{eq:yks}. To that end, by translation invariance we may obviously replace $U_{x_0}$  by the unit cube $[0,1)^d\subset\R^d$ and assume that $\psi\in L^\infty ((0,1)^d)$. Set $k_n:=2\lceil e^n/2\rceil$ so that $k_n\sim e^n$ is even and larger than $e^n$, and divide $[0,1)^d$ into  $(k_n)^d$ copies  of the small cube $[0,1/k_n)^d$, call them $Q_j$, $j=(j_1,\ldots j_d)\in \{1,\ldots ,k_n\}^d.$ Let $\mathcal{A}_n\subset \{1,\ldots ,k_n\}^d$ consist of those $d$-tuples whose all components are odd, so that $\{1,\ldots ,k_n\}^d=\bigcup_{r\in \{0,1\}^d}\bigcup_{j\in \mathcal{A}_n}\{j+r\}$. Our aim is to prove exponential decay for the quantity
		\begin{align*}
		M_n&:= \E \Big| \int_{[0,1)^d} \psi(x)\big(\nu_{n+1,\beta}(x)-\nu_{n,\beta}(x)\big)dx \Big|^p.
		\end{align*}
		We note that by Proposition \ref{pr:split} (i) the fields {$\Delta_n(\cdot)=L_{n+1}-L_n$}  and $\nu_{n,\beta}(\cdot)$ are independent. In order to simplify notation, we assume that our probability space is of product form  $\Omega=\Omega'\times \Omega'', $ $\P=\P'\times\P''$, and the fields $\nu_{n,\beta}$ depend only on $\omega'$, and {$\Delta_n$}  on $\omega''$. We also note that {$\E(\Delta_n(x)^2)=1$} for any $x$. Moreover, by Lemma \ref{le:cascadelike}, the restrictions of the field {$\Delta_n$}  to different cubes $Q_j$ and $Q_{j'},$ $j,j'\in \mathcal{A}_n$ are independent. Since $[0,1)^d$ may be expressed as the disjoint union of $2^d$  translates of the set $\bigcup_{j\in \mathcal{A}_n} Q_j$ , we may use the von Bahr--Esseen inequality \cite{BE} to estimate 
		\begin{align*}
		M_n&\leq  \E 2^{d(p-1)}\sum_{r\in\{0,1\}^d}\Big| \int_{\bigcup_{j\in \mathcal{A}_n} Q_{j+r}} \psi(x)\big(\nu_{n+1,\beta}(x)-\nu_{n,\beta}(x)\big)dx \Big|^p \\
		&= 2^{d(p-1)}\sum_{r\in\{0,1\}^d}\E_{\omega'}\left(\E_{\omega''} \Big| \int_{\bigcup_{j\in \mathcal{A}_n} Q_j} \psi(x)\nu_{n,\beta}(\omega',x)\big(e^{\beta \Delta_{n}(\omega'',x)-\beta^2/2}-1\big)dx \Big|^p\right) \\
		&\le C_p2^{d(p-1)}\sum_{j\in\{1,\dots,k_n\}^d}\E_{\omega'}\left(\E_{\omega''} \Big| \int_{Q_j} \psi(x)\nu_{n,\beta}(\omega',x)\big(e^{\beta \Delta_{n}(\omega'',x)-\beta^2/2}-1\big)dx \Big|^p\right)\\
		&=: C_p2^{d(p-1)}\sum_{j\in\{1,\dots,k_n\}^d} M'_n(j).
		\end{align*}
		We may use H\"older's inequality and translation invariance to bound 
		\begin{align*}
		M'_n(j) &\leq  \|\psi\|_{L^\infty([0,1)^d)}^p(k_n)^{-dp}\E_{\omega'}|\nu_{n,\beta}(\omega',0)|^p\E_{\omega''}\big| e^{\beta \Delta_{n}(\omega'',0)-\beta^2/2}-1\big|^p\\
		&\lesssim  \|\psi\|_{L^\infty([0,1)^d)}^pe^{-ndp}\exp (n(p^2-p)({\rm Re\, }\beta)^2/2+np({\rm Im\, }\beta)^2/2).
		\end{align*} 
		Putting the above estimates together, it follows that $M_n\lesssim \exp( c_\beta n)$, where $c_\beta :=(p^2-p)({\rm Re\, }\beta)^2/2+p({\rm Im\, }\beta)^2/2-d(p-1)$. We may choose $p>1$ so that this quantity is negative assuming that 
		$$
		({\rm Re\, }\beta)^2+\frac{1}{p-1}({\rm Im\, }\beta)^2 <\frac{2d}{p}.
		$$
		One easily checks for each $\beta_0 \in A$ we may choose $p>1$ so that above inequality is satisfied in a neighbourhood $B(\beta_0,\delta)$ of the point $\beta_0$. Finally, the obtained exponential decay of the increments clearly yields \eqref{eq:yks}.
		
		In order to finish the proof of the theorem, we may pick a cover of  $U$ by  neighbourhoods $U_j:=B(x_j,\varepsilon_{x_j})$, $j=1,\ldots$ (here $x_j$ replaces $x_0$ above) so that each compact subset of $U$ intersects only finitely many of the neighbourhoods $U_j$. We choose the related elements $\psi_{x_j}\in C^\infty_c(U_j)$  so that they form a partition of unity in $U$, and form the $H^{-d}(\R^d)$ valued random variables $\eta_{x_j}$ as above. By construction, the random $H^{-d}_{loc}(U)$-valued analytic function on $A$
		$$
		\mu_\beta:=\sum_j\eta_{x_j}(\beta)
		$$
		has all the properties stated in the theorem. Only one thing perhaps needs to be discussed: that  for each fixed $j\geq 1$ the analytic continuation $\eta_{x_j}(\beta)$ is a.s. $\sigma (X)$-measurable, i.e.  a.s. it is a function of the original field $X$. Let $h:A\to\D$ be the Riemann map that fixes origin and maps $\R\cap A$ onto $\R\cap \D$. Set
		$\widetilde \eta_{x_j}(\beta):=\eta_{x_j}(h^{-1}(\beta)).$ Then we have $\eta_{x_j}(\beta) =\widetilde \eta_{x_j}(\widetilde \beta)$ with $\widetilde \beta:=h(\beta)$, and we may instead consider the map $\widetilde\beta \mapsto \widetilde \eta_{x_j}(\widetilde \beta)$, which is analytic on the unit disc. Hence we have the power series expansion
		$$
		\widetilde\eta_{x_j}(\widetilde\beta)=\sum_{k=0}^\infty A_k\widetilde \beta^k,
		$$
		which converges locally uniformly in $H^{-d}(\R^d)$. The coefficients $A_k$ are $H^{-d}(\R^d)$-valued random variables, which a.s. can be computed as 
		\begin{eqnarray*}
			A_k&=&(k!)^{-1}\left[\big(\frac{d}{d\widetilde \beta}\big)^k \widetilde\eta_{x_j}(\widetilde \beta)\right]_{\widetilde\beta =0}
			\:=\;(k!)^{-1}\lim_{m\to\infty} m^k\Big(\sum_{j=0}^k (-1)^{k-j}{ k\choose j} \widetilde\eta_{x_j}(j/m)\Big)\\
			&=&(k!)^{-1}\lim_{m\to\infty} m^k\Big(\sum_{j=0}^k (-1)^{k-j}{ k\choose j}\eta_{x_j}(h^{-1}(j/m))\Big).
		\end{eqnarray*}
		This yields what we wanted, since a.s. all the values  $\eta_{x_j}(h^{-1}(j/m))$ are obtained as in the standard definition of the chaos, and hence are functions of $X$.
	\end{proof}

	\begin{remark}\label{rem:complex}
		The above theorem does not take any stand on what is the optimal Sobolev regularity of the complex chaos $\mu_\beta$, and this will be one of the topics of a sequel to the present paper. \enddef
	\end{remark}
	
	\section{Generalized Onsager inequalities}\label{sec:ichaos}
	
	As the last application of Theorem \ref{th:deco} we shall prove a local  Onsager-type inequality for general log-correlated  fields. As mentioned in the introduction, this result is crucial (see  \cite{JSW}) in order to obtain obtain good enough moment  bounds  for imaginary chaos in general dimensions,  so that the moments determine the chaos uniquely. The idea of the proof is again that it is simple for $\star$-scale invariant fields and extends to general log-correlated fields through Theorem {B}.
	
	\begin{theorem}\label{th:onsager} Assume that $X$ is a log-correlated  field on a domain $U\subset\R^d$ with $0\in U$ and with the same conditions on the covariance as in Theorem \ref{th:deco}. Then, there is a neighbourhood $B_\varepsilon(0)\subset U$  of the origin so that $X$ satisfies an Onsager type inequality in $B_\varepsilon (0)$:\;
		for any
		$n \ge 1$, $q_1, \dots, q_n \in \{-1,1\}$ and distinct $x_1, \dots, x_n \in B_\varepsilon (0)$  it holds that
		\begin{equation}\label{eq:osager}
		-\sum_{1 \le j < k \le n} q_j q_k \E X(x_j) X(x_k) \le \frac{1}{2} \sum_{j=1}^n  \log \frac{1}{\frac{1}{2} \min_{k \neq j} |x_j - x_k|} + C n,
		\end{equation}     
		where $C$ is independent of the points $x_j$ or $n$, but may depend on the neighbourhood $B(0,\varepsilon).$
	\end{theorem}
	\begin{proof} Let $B(0,\varepsilon)$ be a neighbourhood for which  a decomposition $X={L+R}$, given by Theorem \ref{th:deco}(ii), is valid, obtained by some allowed seed covariance function $k$ that is supported on $B(0,1)$ (observe that  any dilatation $k(\lambda_0\cdot)$ is then also supported in $B(0,1)$ for $\lambda_0\geq 1)$. Especially, Lemma \ref{le:cascadelike} applies to the field ${L}$. By independence, it is obviously enough to prove the result separately for both of the fields ${L}$ and ${R}$. Since $C_{{R}}$ is locally bounded, say $|C_{{R}}(x,y)|\leq A$  for $x,y\in B(0,\varepsilon),$ we obtain 
		$$
		-\sum_{1 \le j < k \le n} q_j q_k \E {R}(x_j) {R}(x_k) =-\frac{1}{2}\E \big|\sum_{j=1}^nq_j {R}(x_j)\big|^2+\frac{1}{2}\sum_{j=1}^n\E {R}(x_j)^2\leq nA/2.
		$$
		In turn, to treat the contribution of ${L}$ we may assume that $\varepsilon <1/2$ and denote for each $j\in\{1,\ldots, n\}$ half of the shortest distance to the neighbours by $r_j:=\frac{1}{2}\min_{k\not=j} |x_k-x_j|.$ Define the variables
		$G_j$ for $j=1\ldots n$ by setting $G_j= {L}_{\log(1/r_j)}$, Lemma~\ref{le:cascadelike} implies for distinct $j,k$ that  
		$$
		\E {L}(x_j) {L}(x_k)= \E G_jG_k.
		$$
		By recalling \eqref{eq:covLR}  it thus follows that 
		$$
		-\sum_{1 \le j < k \le n} q_j q_k \E {L}(x_j){L}(x_k) =-\frac{1}{2}\E \big|\sum_{j=1}^nq_jG_j\big|^2+\frac{1}{2}\sum_{j=1}^n\E G_j^2\leq \frac{1}{2}\sum_{j=1}^n\log(1/r_j).
		$$
		Put together, the claim follows with $C=A/2.$
	\end{proof}

\end{document}